\documentclass[smallextended]{svjour3}       
\smartqed  
\usepackage{graphicx}
\usepackage{amsmath}
\usepackage{amssymb}
\usepackage{url}
\usepackage{hyperref}
\usepackage{algorithm}
\usepackage{algpseudocode}
\usepackage{xfrac}
\usepackage{color}
\usepackage[dvipsnames,table]{xcolor}
\usepackage{booktabs} 

\newcommand{\R}{\mathbb{R}}
\newcommand{\Rext}{\mathbb{R}\cup\left\{+\infty\right\}}

\newcommand{\set}[1]{\left\{#1\right\}}
\newcommand{\norm}[1]{\Vert#1\Vert}
\newcommand{\dist}[1]{\mathrm{dist}\left(#1\right)}
\newcommand{\Eproof}{\hfill $\square$}
\newcommand{\argmin}{\mathrm{arg}\!\min}
\newcommand{\argmax}{\mathrm{arg}\!\max}
\newcommand{\prox}{\mathrm{prox}}
\newcommand{\dom}[1]{\mathrm{dom}(#1)}
\newcommand{\xb}{{x}}
\newcommand{\yb}{{y}}
\newcommand{\zb}{{z}}
\newcommand{\ub}{{u}}
\newcommand{\vb}{{v}}
\newcommand{\rb}{{r}}

\newcommand{\bb}{{b}}

\newcommand{\Ab}{{A}}
\newcommand{\Bb}{{B}}

\renewcommand{\sb}{{s}}
\newcommand{\iprod}[1]{\left\langle #1\right\rangle}
\newcommand{\iprodb}[1]{\big\langle #1\big\rangle}
\newcommand{\iprods}[1]{\langle #1\rangle}
\newcommand{\Xb}{{X}}

\newcommand{\Ac}{\mathcal{A}}
\newcommand{\Uc}{\mathcal{U}}

\newcommand{\Xc}{\mathcal{X}}

\newcommand{\Bc}{\mathcal{B}}

\newcommand{\Kc}{\mathcal{K}}

\newcommand{\Fopt}{F^{\star}}
\newcommand{\Fc}{\mathcal{F}}

\newcommand{\xopt}{{x}^{\star}}

\newcommand{\uast}{{u}^{\ast}}

\newcommand{\xhat}{\hat{{x}}}
\newcommand{\ubar}{\bar{{u}}}

\newcommand{\xbar}{\bar{{x}}}

\begin{document}

\title{Adaptive Smoothing Algorithms for Nonsmooth Composite Convex Minimization
}


\author{Quoc Tran-Dinh
}


\institute{Quoc Tran-Dinh \at
		Department of Statistics and Operations Research\\
		University of North Carolina at Chapel Hill (UNC), USA.\\
		Email: \url{quoctd@email.unc.edu}
}
\vspace{-3ex}

\date{Received: date / Accepted: date}

\maketitle

\vspace{-4ex}
\begin{abstract}
We propose an adaptive smoothing algorithm based on Nesterov's smoothing technique in \cite{Nesterov2005c} for solving ``fully" nonsmooth composite convex optimization problems.
Our method combines  both Nesterov's accelerated proximal gradient scheme and a new homotopy strategy for smoothness parameter.
By an appropriate choice of smoothing functions, we develop a new algorithm that has  the $\mathcal{O}\left(\frac{1}{\varepsilon}\right)$-worst-case iteration-complexity 
while preserves the same complexity-per-iteration as in Nesterov's method and allows one to automatically update the smoothness parameter at each iteration. 
Then, we customize our algorithm to solve four special cases that cover various applications.
We also specify our algorithm to solve constrained convex optimization problems and show its convergence guarantee on a primal sequence of iterates. 
We demonstrate our algorithm through three numerical examples and compare it with other related algorithms.
\end{abstract}


\vspace{-2ex}
\section{Introduction}\label{sec:intro}
\vspace{-2ex}
This paper develops new smoothing optimization methods for solving the following ``fully'' nonsmooth composite convex minimization problem:
\begin{equation}\label{eq:com_min}
F^{\star} := \min_{\xb\in\R^p}\Big\{ F(\xb) := f(\xb) + g(\xb) \Big\},
\end{equation}
where $g:\R^p\to\Rext$ is a proper, closed and convex function, and $f : \R^p\to\Rext$ is a convex function defined by the following max-structure:
\begin{equation}\label{eq:fx_structure}
f(\xb) := \max_{\ub\in\R^n}\Big\{ \iprods{\xb, \Ab\ub} - \varphi(\ub) : \ub\in\Uc \Big\}.
\end{equation}
Here, $\varphi : \R^n\to\Rext$ is a proper, closed and convex function, and $\Uc$ is a nonempty, closed, and convex set in $\R^n$, and $\Ab\in\R^{p\times n}$ is given.

Clearly, any proper, closed and convex function $f$ can be written as \eqref{eq:fx_structure} using its Fenchel conjugate $f^{*}$, i.e., $f(\xb) := \sup\set{\iprods{\xb, \ub} - f^{*}(\ub) : \ub\in\dom{f^{*}}}$.
Hence, the max-structure \eqref{eq:fx_structure} does not restrict the applicability of the template \eqref{eq:com_min}.
Moreover, \eqref{eq:com_min} also directly models many practical applications in signal and image processing, machine learning, statistics and data sciences, see, e.g., \cite{Beck2009,BenTal2001,Boyd2011,Combettes2011a,Nesterov2007,Parikh2013,Tran-Dinh2013a} and the references quoted therein.

While the first term $f$ is nonsmooth, the second term $g$ remains unspecified.
On the one hand, we can assume that $g$ is smooth and its gradient is Lipschitz continuous. 
On the other hand, $g$ can be nonsmooth, but it is equipped with a ``tractable'' proximity operator defined as follows:
$g$ is said to be \textit{tractably proximal} if its proximal operator
\vspace{-0.5ex}
\begin{equation}\label{eq:prox_g}
\prox_{g}(\xb) := \argmin_{\yb}\set{ g(\yb) + (1/2)\norm{\yb - \xb}^2 : \yb\in\dom{g}},
\vspace{-0.75ex}
\end{equation}
can be computed ``efficiently'' (e.g., by a closed form or by polynomial time algorithms). 
In general, computing $\prox_g$ requires to solve the strongly convex problem \eqref{eq:prox_g}, 
but in many cases, this operator can be obtained in a closed form or by a low-cost polynomial algorithm. 
Examples of such convex functions can be found in the literature including \cite{Bauschke2011,Combettes2011a,Parikh2013}.

Solving nonsmooth convex optimization problems remains challenging, especially when none of the two nonsmooth terms $f$ and $g$ is equipped with a tractable proximity operator. 
Existing nonsmooth convex optimization approaches such as subgradient-type descent algorithms, dual averaging strategies, bundle-level techniques or derivative-free methods are often used to solve general nonsmooth convex problems. 
However, these methods suffer a slow convergence rate (resp.,  $\mathcal{O}\left(\frac{1}{\epsilon^2}\right)$ - worst-case iteration-complexity). In addition, they are sensitive to the algorithmic parameters such as stepsizes~\cite{Nesterov2004}.

In his pioneering work \cite{Nesterov2005c}, Nesterov shown that one can solve the nonsmooth structured convex minimization problem \eqref{eq:com_min} within $\mathcal{O}\left(\frac{1}{\epsilon}\right)$ iterations.
This method combines a proximity smoothing technique and Nesterov's accelerated gradient scheme \cite{Nesterov1983} to achieve the optimal worst-case iteration-complexity, which is much better than the $\mathcal{O}\left(\frac{1}{\epsilon^2}\right)$-worst-case iteration complexity in nonsmooth optimization methods.

Motivated by \cite{Nesterov2005c}, Nesterov  and many other researchers have proposed different algorithms using such a proximity smoothing method to solver other problems, to improve Nesterov's original algorithm or customize his algorithm to specific applications, see, e.g., \cite{baes2009smoothing,Becker2011b,Becker2011a,chen2014first,Goldfarb2012,Necoara2008,Nedelcu2014,Nesterov2005d,Nesterov2007d,TranDinh2012a}.
In \cite{Beck2012a}, Beck and Teboulle generalized Nesterov's smoothing technique to a generic framework, where they discussed the advantages and disadvantages of smoothing techniques. In addition, they also illustrated the numerical efficiency between smoothing techniques and proximal-type methods.
In \cite{argyriou2014hybrid,orabona2012prisma}, the authors studied smoothing techniques for the sum of three convex functions, where one term is Lipschitz gradient, while the others are nonsmooth.
In \cite{boct2012variable}, a variable smoothing method was proposed, which possesses the  $\mathcal{O}\left(\frac{\ln(k)}{k}\right)$-convergence rate. 
This convergence rate is worse than the one in \cite{Nesterov2005c}. However, as a compensation, the smoothness parameter is updated at each iteration.
In addition, their method uses special quadratic proximity functions, while smooths both $f$ and $g$ under their Lipschitz  continuity assumption.

In \cite{Nesterov2005d}, Nesterov introduced an excessive gap technique, which  requires both primal and dual schemes using two smoothness parameters. 
It symmetrically updates one parameter at each iteration.
Nevertheless, this method uses different assumptions than our method.
Other primal-dual methods studied in, e.g., \cite{Bot2013,Devolder2012} use double smoothing techniques to solve \eqref{eq:com_min}, but only achieve $\mathcal{O}\left(\frac{1}{\varepsilon}\log\left(\frac{1}{\varepsilon}\right)\right)$-worst-case iteration-complexity.

Our approach in this paper is also based on Nesterov's smoothing technique in \cite{Nesterov2005c}.
To clarify the differences between our method and \cite{Nesterov2005d,Nesterov2005c}, let us first briefly present Nesterov's smoothing technique in \cite{Nesterov2005c} applying to \eqref{eq:com_min}.

Recall that a convex function $b_{\Uc} : \R^n\to\R$ is a proximity function of $\Uc$ if it is continuous, and strongly convex with the convexity parameter $\mu_b > 0$ and $\Uc\subseteq \dom{b_{\Uc}}$.
We define
\vspace{-0.5ex}
\begin{equation*}
\ubar^c := \argmin_{\ub}\set{ b_{\Uc}(\ub) : \ub\in\Uc} ~~~\text{and}~~D_{\Uc} := \sup_{\ub}\set{ b_{\Uc}(\ub) : \ub\in\Uc} \in [0, +\infty).
\vspace{-0.5ex}
\end{equation*}
Here, $\ubar^c$ and $D_{\Uc}$ are called the prox-center and prox-diameter of $\Uc$ w.r.t. $b_{\Uc}$, respectively. 
Without loss of generality, we can assume that $b_{\Uc}(\ubar^c) = 0$ and $\mu_b = 1$. Otherwise, we just rescale and shift it.

As shown in \cite{Nesterov2005c}, given  $\gamma > 0$ and $b_{\Uc}$, we can  approximate $f$ by $f_{\gamma}$ as
\vspace{-0.5ex}
\begin{equation}\label{eq:f_gamma}
f_{\gamma}(\xb) := \max_{\ub}\set{ \iprods{\xb, \Ab\ub} - \varphi(\ub) - \gamma b_{\Uc}(\ub) : \ub\in\Uc},
\vspace{-0.5ex}
\end{equation}
where $\gamma$ is called a smoothness parameter.
Since $f_{\gamma}$ is smooth and has Lipschitz gradient,  one can apply accelerated proximal gradient methods \cite{Beck2009,Nesterov2007} to minimize the sum $f_{\gamma}(\cdot) + g(\cdot)$.
Using such methods, we can eventually guarantee
\vspace{-0.5ex}
\begin{equation}\label{eq:nes_bound}
F(\xb^k) - \Fopt \leq \min_{\gamma > 0}\set{\frac{2\norm{\Ab}^2R_0^2}{\gamma(k+1)^2} + \gamma D_{\Uc}} = \frac{2\sqrt{2}\norm{\Ab}R_0\sqrt{D_{\Uc}}}{(k+1)},
\vspace{-0.5ex}
\end{equation}
where $\set{\xb^k}$ is the underlying sequence generated by the accelerated proximal-gradient method, see \cite{Nesterov2005c}, and $R_0 := \norm{\xb^0 - \xopt}$. 
To achieve an $\varepsilon$-solution $\xb^k$ such that $F(\xb^k) - \Fopt\leq\varepsilon$, we set $\gamma \equiv \gamma^{*} :=  \frac{\varepsilon}{2D_{\Uc}}$ at its optimal value.
Hence, the algorithm requires at most $k_{\max} := \left\lfloor 2\sqrt{2}\norm{\Ab}R_0\sqrt{D_{\Uc}}\varepsilon^{-1}\right\rfloor$ iterations.

\vspace{-2.5ex}
\paragraph{Our approach:}
The original smoothing algorithm in \cite{Nesterov2005c} has three computational disadvantages even with  the optimal choice $\gamma^{*} \!:=\! \frac{\varepsilon}{2D_{\Uc}}$ of $\gamma$.
\begin{itemize}
\item[\textrm{(a)}] It requires the prox-diameter $D_{\Uc}$ of $\Uc$ to determine $\gamma^{\ast}$, which may be expensive to estimate when $\Uc$ is complicated.

\item[\textrm{(b)}] If $\varepsilon$ is small and $D_{\Uc}$ is large, then $\gamma^{*}$ is small, and hence, the strong convexity parameter of \eqref{eq:f_gamma} is small.
Algorithms for solving \eqref{eq:f_gamma} have slow convergence speed. 

\item[\textrm{(c)}] The Lipschitz constant of $\nabla{f_{\gamma}}$ is $\gamma^{-1}\Vert A\Vert^2 = \Vert A\Vert^2D_{\Uc}\varepsilon^{-1}$, which  is large. This  leads to a small step-size of $\sfrac{\varepsilon}{(\Vert A\Vert^2D_{\Uc})}$ in the accelerated proximal-gradient algorithm and hence, can have a slow convergence.
\end{itemize} 
Our approach is briefly presented as follows.
We first choose a smooth proximity function $b_{\Uc}$ instead of a general one. We assume that $\nabla{b}_{\Uc}$ is $L_b$-Lipschitz continuous with the Lipschitz constant $L_b \geq \mu_b = 1$.
Then, we define $f_{\gamma}(\xb)$ as in \eqref{eq:f_gamma}, which is a smoothed approximation to $f$ as above.

We design a smoothing accelerated proximal-gradient algorithm that can updates $\gamma$ from $\gamma_k$ to $\gamma_{k\!+\!1}$ at each iteration so that $\gamma_{k\!+\!1} < \gamma_k$ by performing only \emph{one} accelerated proximal-gradient step \cite{Beck2009,Nesterov2007} to minimize the sum $F_{\gamma_{k\!+\!1}} :=  f_{\gamma_{k\!+\!1}} + g$ for each value $\gamma_{k\!+\!1}$ of $\gamma$.
We prove that the sequence of the objective residuals, $\set{F(\xb^k) - \Fopt}$, converges to zero up to the $\mathcal{O}\left(\frac{1}{k}\right)$-rate.

\vspace{-2.0ex}
\paragraph{Our  contributions:}
Our main contributions can be summarized as follows:
\begin{itemize}
\vspace{-1ex}
\item[$\mathrm{(a)}$] 
We propose using a smooth proximity function  to smooth the max-structure objective function $f$ in \eqref{eq:fx_structure}, and develop a new smoothing algorithm, Algorithm \ref{alg:A1}, based on the accelerated proximal-gradient method to adaptively update the smoothness parameter in a heuristic-free fashion.

\item[$\mathrm{(b)}$] 
We prove up to the $\mathcal{O}\left(\frac{1}{\varepsilon}\right)$-worst-case iteration-complexity for our algorithm as in \cite{Nesterov2005c} to achieve an $\varepsilon$-solution, i.e., $F(\xb^k) - \Fopt \leq \varepsilon$.
Especially, with the quadratic proximity function $b_{\Uc}(\cdot) := (1/2)\norm{\cdot - \bar{\ub}^c}^2$, our algorithm achieve exactly the $\mathcal{O}\left(\frac{1}{\varepsilon}\right)$-worst-case iteration-complexity as in \cite{Nesterov2005c}.

\item[$\mathrm{(c)}$] 
We customize our algorithm to handle four important special cases that have a great practical impact in many applications.

\item[$\mathrm{(d)}$]
We specify our algorithm to solve constrained convex minimization problems, and propose an averaging scheme to recover an approximate primal  solution with a rigorous convergence guarantee.
\end{itemize}
\vspace{-1ex}
\noindent
From a practical point of view, we believe that the proposed algorithm can overcome three disadvantages mentioned previously in the original smoothing algorithm in \cite{Nesterov2005c}. 
However, our condition $L_b = 1$ on the choice of proximity functions may lead to some limitation of the proposed algorithm for exploiting further the structures of the constrained set $\Uc$.
Fortunately, we can identify several important settings in Section \ref{sec:exploit_struct}, where we can eliminate this disadvantage. 
Such classes of problems cover several applications in image processing, compressive sensing, and monotropic programming \cite{Bauschke2011,Combettes2011a,Parikh2013,Yang2011}.

\vspace{-2ex}
\paragraph{Paper organization:}
The rest of this paper is organized as follows.
Section~\ref{sec:prel} briefly discusses our smoothing technique.
Section~\ref{sec:main_alg} presents our main algorithm, Algorithm \ref{alg:A1}, and proves its convergence guarantee.
Section~\ref{sec:exploit_struct}  handles four special but important cases of \eqref{eq:com_min}. 
Section~\ref{sec:app} specializes our algorithm to solve constrained convex minimization problems.
Preliminarily numerical examples are given in Section~\ref{sec:num_exp}.
For clarity of presentation, we move the long and technical proofs to the appendix.

\vspace{-3ex}
\paragraph{Notation and terminology:}
We work on the real spaces $\R^p$ and $\R^n$, equipped with the standard inner product $\iprods{\cdot, \cdot}$ and the Euclidean $\ell_2$-norm $\norm{\cdot}$.
Given a proper, closed, and convex function $g$, we use $\dom{g}$ and $\partial{g}(\xb)$ to denote its domain and its subdifferential at $\xb$, respectively. 
If $g$ is differentiable, then $\nabla{g}(\xb)$ stands for its gradient at $\xb$.

We denote $f^{*}(\sb) := \sup\set{\iprod{\sb, \xb} - f(\xb) : \xb\in\dom{f}}$, the Fenchel conjugate of $f$.
For a given set $\Xc$, $\delta_{\Xc}(\xb) := 0$ if $\xb\in\Xc$ and $\delta_{\Xc}(\xb) := +\infty$, otherwise, defines the indicator function of $\Xc$.
For a smooth function $f$, we say that $f$ is $L_f$-smooth if for any $\xb, \tilde{\xb}\in\dom{f}$, we have $\Vert\nabla{f}(\xb) - \nabla{f}(\tilde{\xb})\Vert \leq L_f\norm{\xb - \tilde{\xb}}$, where $L(f) := L_f \in [0, \infty)$. 
We denote by $\mathcal{F}_{L}^{1,1}$ the class of all $L_f$-smooth and convex functions $f$.
We also use $\mu_f \equiv \mu(f)$ for the strong convexity parameter of a convex function $f$.
For a given symmetric matrix $\Xb$, $\lambda_{\min}(\Xb)$ and $\lambda_{\max}(\Xb)$ denote its smallest and largest eigenvalues of $\Xb$, respectively; and $\mathrm{cond}(\Xb)$ is the condition number of $\Xb$. 
Given a nonempty, closed and convex set $\Xc$, $\dist{\xb,\Xc}$ denotes the distance from $\xb$ to $\Xc$.

\vspace{-3.5ex}
\section{Smoothing techniques via smooth proximity functions}\label{sec:prel}
\vspace{-2ex}
Let $b_{\Uc}$ be a prox-function of the nonempty, closed and convex set $\Uc$ with the strong convexity parameter $\mu_b = 1$. 
In addition, $b_{\Uc}$ is smooth on $\Uc$, and its gradient $\nabla{b_{\Uc}}$ is Lipschitz continuous with the Lipschitz constant $L_b \geq \mu_b = 1$.
In this case, $b_{\Uc}$  is said to be $L_b$-smooth.
As a default example, $b_{\Uc}(\cdot) := (1/2)\norm{\cdot - \ubar^c}^2$ for fixed $\ubar^c\in\Uc$ satisfies our assumptions with $L_b = \mu_b = 1$.
Let $\ubar^c$ be the $b$-prox-center point of $\Uc$, i.e., $\ubar^c := \argmin_{\ub}\set{b_{\Uc}(\ub) : \ub\in\Uc}$. 
Without loss of generality, we can assume that  $b_{\Uc}(\ubar^c) = 0$. 
Otherwise, we consider $\bar{b}_{\Uc}(\ub) := b_{\Uc}(\ub) - b_{\Uc}(\ubar^c)$.

Given a convex function $\varphi^{\ast}(\zb) := \max_{\ub}\set{\iprods{\zb, \ub} - \varphi(\ub) : \ub\in\Uc}$, we define a smoothed approximation of$ \varphi^{\ast}$ as
\begin{equation}\label{eq:b_cond}
\varphi^{\ast}_{\gamma}(\zb) := \max_{\ub\in\Uc}\set{\iprods{\zb, \ub} - \varphi(\ub) - \gamma b_{\Uc}(\ub)},
\end{equation}
where $\gamma > 0$ is a smoothness parameter.
We note that $\varphi^{\ast}$ is not a Fenchel conjugate of $\varphi$ unless $\Uc = \dom{\varphi}$. 
We denote by $\uast_{\gamma}(\xb)$ the unique optimal solution of the strongly concave maximization problem~\eqref{eq:b_cond}, i.e.:
\begin{equation}\label{eq:uast_beta}
\uast_{\gamma}(\zb) \in \mathrm{arg}\!\max_{\ub}\set{ \iprods{\zb, \ub} - \varphi(\ub) - \gamma b_{\Uc}(\ub) : \ub\in\Uc}.
\end{equation}
We also define $D_{\Uc} := \sup_{\ub}\set{b_{\Uc}(\ub) : \ub\in\Uc\cap\dom{\varphi}}$ the $b$-prox diameter of $\Uc$.
If $\Uc$ or $\dom{\varphi}$  is bounded, then $D_{\Uc} \in [0, +\infty)$.

Associated with $\varphi^{\ast}_{\gamma}$, we  consider a smoothed function for  $f$  in \eqref{eq:fx_structure} as
\begin{equation}\label{eq:f_beta}
f_{\gamma}(\xb) := \varphi^{\ast}_{\gamma}(\Ab^{\top}\xb) = \max_{\ub}\set{ \iprods{\Ab^{\top}\xb, \ub} - \varphi(\ub) - \gamma b_{\Uc}(\ub) : \ub\in\Uc}.
\end{equation}
Then, the following lemma summaries the properties of the smoothed function $\varphi^{\ast}_{\gamma}$ defined by \eqref{eq:b_cond} and $f_{\gamma}$ defined by \eqref{eq:f_beta}, whose proof can be found in \cite{Tran-Dinh2014a}.

\begin{lemma}\label{le:fbeta_smoothness}
The function $\varphi^{\ast}_{\gamma}$ defined by \eqref{eq:b_cond} is convex and smooth. Its gradient is given by $\nabla{\varphi^{\ast}_{\gamma}}(\zb) := \ub^{\ast}_{\gamma}(\zb)$ which is Lipschitz continuous with the Lipschitz constant $L_{\varphi_{\gamma}^{\ast}} := \gamma^{-1}$. Consequently, for any $\zb,\bar{\zb}\in\R^n$, we have
\begin{equation}\label{eq:pro_cond1}
\frac{\gamma}{2}\norm{\ub^{\ast}_{\gamma}(\zb) - \ub^{\ast}_{\gamma}(\bar{\zb})}^2 \leq \varphi_{\gamma}^{\ast}(\zb) - \varphi_{\gamma}^{\ast}(\bar{\zb}) - \iprods{\nabla{\varphi_{\gamma}^{\ast}}(\zb), \zb - \bar{\zb}} \leq \frac{1}{2\gamma}\norm{\zb - \bar{\zb}}^2.
\end{equation}
For fixed $\zb\in\R^n$,  $\varphi_{\gamma}^{\ast}(\zb)$ is convex w.r.t. $\gamma \in\R_{++}$, and
\begin{equation}\label{eq:key_est2}
\varphi^{\ast}_{\gamma}(\zb) - (\hat{\gamma} - \gamma)b_{\Uc}(\uast_{\gamma}(\zb)) \leq \varphi^{\ast}_{\hat{\gamma}}(\zb), ~~\forall \gamma,\hat{\gamma}\in\R_{++}.
\end{equation}
As a consequence, $f_{\gamma}$ defined by \eqref{eq:f_beta} is convex and smooth. Its gradient is given by $\nabla{f_{\gamma}}(\xb) = \Ab\uast_{\gamma}(\Ab^{\top}\xb)$, which is Lipschitz continuous with the Lipschitz constant $L_{f_{\gamma}} := \gamma^{-1}\norm{\Ab}^2$.
In addition, we also have
\begin{equation}\label{eq:key_est1}
f_{\gamma}(\xb) \leq f(\xb) \leq f_{\gamma}(\xb) + \gamma D_{\Uc}, ~~\forall \xb\in\R^p.
\end{equation}
\end{lemma}
We emphasize that Lemma~\ref{le:fbeta_smoothness} provides key properties to analyze the complexity of our algorithm in the next setions.

\vspace{-3ex}
\section{The adaptive smoothing algorithm and its convergence}\label{sec:main_alg}
\vspace{-2ex}
Associated with \eqref{eq:com_min}, we consider its smoothed composite convex problem as
\begin{equation}\label{eq:smoothed_com_min}
F_{\gamma}^{\star} := \min_{\xb\in\R^p}\set{ F_{\gamma}(\xb) := f_{\gamma}(\xb) + g(\xb)}.
\end{equation}
Similar to \cite{Nesterov2005c}, the main step of Nesterov's accelerated proximal-gradient scheme \cite{Beck2009,Nesterov2007} applied to  the smoothed problem~\eqref{eq:smoothed_com_min} is expressed as follows:
\begin{equation}\label{eq:prox_grad_step}
\begin{array}{ll}
\xb^{k\!+\!1} &:= \prox_{\beta g}\left(\xhat^k - \beta \nabla{f}_{\gamma}(\xhat^k)\right)\vspace{1ex}\\
&\equiv \displaystyle\argmin_{\xb\in\R^p}\Big\{ g(\xb) + 
\frac{1}{2\beta}\big\Vert \xb -  \big(\xhat^k - \beta\Ab\uast_{\gamma}(\Ab^{\top}\xhat^k)\big)\big\Vert^2\Big\},
\end{array}
\end{equation}
where $\hat{\xb}^k$ is given, and $\beta > 0$ is a given step size, which will be chosen later.

The following lemma provides a descent property of the proximal-gradient step \eqref{eq:prox_grad_step}, whose proof can be found in Appendix~\ref{apdx:le:descent_inequality}. 

\begin{lemma}\label{le:descent_inequality}
Let $\xb^{k\!+\!1}$ be  generated by \eqref{eq:prox_grad_step}. Then, for any $\xb\in\R^p$, we have
\begin{equation}\label{eq:key_est21}
\begin{array}{ll}
F_{\gamma}(\xb^{k\!+\!1}) &\leq \hat{\ell}^k_{\gamma}(\xb) + \frac{1}{\beta}\iprods{\xb^{k\!+\!1} - \xhat^k, \xb -\xhat^k} - \frac{1}{2}\left(\frac{2}{\beta} - \frac{\norm{\Ab}^2}{\gamma}\right)\norm{\xhat^k - \xb^{k\!+\!1}}^2, 
\end{array}
\end{equation}
where 
\begin{equation}\label{eq:key_est21a}
\begin{array}{ll}
\hat{\ell}^k_{\gamma}(\xb) &:= f_{\gamma}(\xhat^k) + \iprods{\nabla{f_{\gamma}}(\xhat^k), \xb - \xhat^k} + g(\xb) \vspace{1ex}\\
& \leq F_{\gamma}(\xb) - \frac{\gamma}{2}\norm{\ub^{\ast}_{\gamma}(\Ab^{\top}\xb) - \ub^{\ast}_{\gamma}(\Ab^{\top}\xhat^k)}^2.
\end{array}
\end{equation}
\end{lemma}
We now adopt the accelerated proximal-gradient scheme (FISTA) in  \cite{Beck2009} to solve \eqref{eq:smoothed_com_min} using an adaptive step-size $\beta_{k\!+\!1} := \frac{\gamma_{k\!+\!1}}{\norm{\Ab}^2}$, which becomes
\begin{equation}\label{eq:acc_grad_scheme}
\left\{\begin{array}{ll}
\xhat^k     &:= (1-\tau_k)\xb^k + \tau_k\tilde{\xb}^k\vspace{0.7ex}\\
\xb^{k\!+\!1} &:= \prox_{\beta_{k\!+\!1}g}\left(\xhat^k - \beta_{k\!+\!1}\nabla{f}_{\gamma_{k\!+\!1}}(\xhat^k)\right)\vspace{1.0ex}\\
\tilde{\xb}^{k\!+\!1} &:= \tilde{\xb}^k -\frac{1}{\tau_k}(\xhat^k - \xb^{k\!+\!1}),
\end{array}\right.
\end{equation}
where $\gamma_{k\!+\!1} > 0$ is the smoothness parameter, and $\tau_k \in (0, 1]$. 

By letting $t_k := \frac{1}{\tau_k}$, we can eliminate  $\tilde{\xb}^k$ in \eqref{eq:acc_grad_scheme} to obtain a compact version
\begin{equation}\label{eq:acc_grad_scheme2}
\left\{\begin{array}{ll}
\xb^{k\!+\!1} &:= \prox_{\beta_{k\!+\!1}g}\left(\xhat^k - \beta_{k\!+\!1}\nabla{f}_{\gamma_{k\!+\!1}}(\xhat^k)\right)\vspace{1.2ex}\\
\xhat^{k\!+\!1} & := \xb^{k\!+\!1} + \frac{t_k-1}{t_{k\!+\!1}}(\hat{\xb}^k - \xb^k).
\end{array}\right.
\end{equation}
The following lemma provides a key estimate to prove the convergence of the scheme~\eqref{eq:acc_grad_scheme} (or \eqref{eq:acc_grad_scheme2}), whose proof can be found in Appendix \ref{apdx:le:key_estimate1}.

\begin{lemma}\label{le:key_estimate1}
Let $\set{(\xb^k, \tilde{\xb}^k, \tau_k,\gamma_k)}$ be the sequence generated by \eqref{eq:acc_grad_scheme}. 
Then
\begin{equation}\label{eq:key_estimate31}
F_{\gamma_{k\!+\!1}}(\xb^{k\!+\!1}) \leq (1 \!-\! \tau_k)F_{\gamma_{k}}(\xb^k) \!+\! \tau_kF(\xb) \!+\! \frac{\norm{\Ab}^2\tau_k^2}{2\gamma_{k\!+\!1}}\left[\norm{\tilde{\xb}^k \!\!-\! \xb}^2 \!-\! \norm{\tilde{\xb}^{k\!+\!1}  \!\!-\! \xb}^2\right] \!-\! R_k,
\end{equation}
for any $x\in\R^p$ and $R_k$ is defined by
\begin{align}\label{eq:Rk_term}
\begin{array}{ll}
R_k & := \tau_k\gamma_{k\!+\!1}b_{\Uc}(\uast_{\gamma_{k\!+\!1}}(\Ab^{\top}\xhat^k)) -  (1-\tau_k)(\gamma_k - \gamma_{k\!+\!1})b_{\Uc}(\uast_{\gamma_{k\!+\!1}}(\Ab^{\top}\xb^k))\vspace{1.25ex}\\
&+ \frac{(1 - \tau_k)\gamma_{k\!+\!1}}{2}\norm{\uast_{\gamma_{k\!+\!1}}(\Ab^{\top}\xb^k) \!-\! \uast_{\gamma_{k\!+\!1}}(\Ab^{\top}\xhat^k)}^2.
\end{array}
\end{align}
Moreover, the quantity $R_k$ is bounded from below by
\begin{equation}\label{eq:Rk_lower_bound}
\begin{array}{ll}
R_k &\geq \frac{1}{2}(1-\tau_k)\left[\gamma_{k\!+\!1}\tau_k - L_b(\gamma_k - \gamma_{k\!+\!1})\right]b_{\Uc}(\ub^{*}_{\gamma_{k\!+\!1}}(\Ab^{\top}\xb^k)).
\end{array}
\end{equation}
\end{lemma}

Next, we show \textit{one} possibility for updating $\tau_k$ and $\gamma_k$, and provide an upper bound for $F_{\gamma_k}(\xb^k) - \Fopt$.
The proof of this lemma is moved to Appendix \ref{apdx:le:key_estimate2}.

\begin{lemma}\label{le:key_estimate2}
Let us choose $\tilde{\xb}^0 := \xb^0 \in\dom{F}$, $\gamma_1 > 0$, and an arbitrary constant $\bar{c} \geq 1$.
If the parameters $\tau_k$ and $\gamma_k$  are updated by
\vspace{-0.5ex}
\begin{equation}\label{eq:param_update_rules}
\tau_k := \frac{1}{k+\bar{c}} ~~~\text{and}~~~\gamma_{k\!+\!1} := \frac{\gamma_1\bar{c}}{k+\bar{c}}, 
\vspace{-0.5ex}
\end{equation}
then the quantity $R_k$ defined by \eqref{eq:Rk_term} and $\{ (\tau_k, \gamma_k)\}$ satisfy 
\begin{equation}\label{eq:Rk_lower_bound2}
\frac{\gamma_{k\!+\!1}}{\tau_k^2}R_k \geq -\frac{\gamma_1^2\bar{c}^2\left[(L_b\!-\!1)(k \!+\! \bar{c}) \!+\! 1\right]}{(k \!+\! \bar{c})^2}D_{\Uc}~~~~\text{and}~~~~\frac{(1-\tau_k)\gamma_{k\!+\!1}}{\tau_k^2} = \frac{\gamma_k}{\tau_{k-1}^2}.
\end{equation}
Moreover, the following estimate holds
\vspace{-0.5ex}
\begin{equation}\label{eq:key_estimate2_new}
{\!\!\!\!\!\!}F_{\gamma_{k\!+\!1}}(\xb^{k\!+\!1}) \!-\! \Fopt \!\leq\! \frac{\tau_{k}^2}{\gamma_{k\!+\!1}}\!\!\left[\!\frac{(1\!-\!\tau_0)\gamma_1}{\tau_0^2}(F_{\gamma_0}(\xb^0) \!-\! \Fopt) \!+\! \frac{\norm{\Ab}^2}{2}\norm{\xb^0 \!\!-\! \xopt}^2 \!+\! S_kD_{\Uc}\!\right]\!,{\!\!\!\!}
\vspace{-0.5ex}
\end{equation}
where 
\begin{equation}\label{eq:Sk_bound}
\begin{array}{ll}
{\!\!\!}S_k := \gamma_1^2\bar{c}^2\sum_{i=0}^k\left[\frac{(L_b \!-\!1)}{(i \!+\! \bar{c})} \!+\!  \frac{1}{(i \!+\! \bar{c})^2}\right]  \leq \gamma_1^2\bar{c}^2(L_b\!-\!1)\left(\ln(k\!+\!\bar{c}) \!+\! 1\right) \!+\! \gamma_1^2(\bar{c}\!+\!1).
\end{array}
\end{equation}
In particular, if we choose $b_{\Uc}$ such that $L_b = 1$, then $S_k \leq \gamma_1^2(\bar{c}+1)$.
\end{lemma}

By \eqref{eq:param_update_rules}, the second line of \eqref{eq:acc_grad_scheme2} reduces to $\hat{\xb}^{k\!+\!1} := \xb^{k\!+\!1} + \left(\frac{k+\bar{c}-1}{k+\bar{c}+1}\right)(\xb^{k\!+\!1} - \xb^k)$.
Using this step into \eqref{eq:acc_grad_scheme2} and combining the result with the update rule \eqref{eq:param_update_rules}, we can present our algorithm for solving \eqref{eq:com_min} as in Algorithm \ref{alg:A1}.

\begin{algorithm}[ht!]\caption{(\textit{Adaptive Smoothing Proximal-Gradient Algorithm})}\label{alg:A1}
\begin{normalsize}
\begin{algorithmic}[1]
\Statex{\hskip-4ex}\textbf{Initialization:} 
\State\label{step:1}Choose $\gamma_1 > 0$, $\bar{c} \geq 1$ and $\xb^0 \in\R^p$. Set $\xhat^0 := \xb^0$.
\Statex{\hskip-4ex}\textbf{Iteration:} \textbf{For}~{$k=0$ {\bfseries to} $k_{\max}$, \textbf{perform:}}
\State\label{step:2}Solve the following strongly concave maximization subproblem
\begin{equation*}
\hat{\ub}^{\ast}_{\gamma_{k\!+\!1}}(\xhat^k) := \argmax_{\ub\in\Uc}\Big\{\iprods{\xhat^k, \Ab\ub} - \varphi(\ub) - \gamma_{k\!+\!1}b_{\Uc}(\ub )\Big\}.
\end{equation*}
\State\label{step:3}Perform the following proximal-gradient step with $\beta_{k\!+\!1} := \frac{\gamma_{k\!+\!1}}{\norm{\Ab}^2}$:
\begin{equation*}
\xb^{k\!+\!1} := \prox_{\beta_{k\!+\!1}g}\left(\xhat^k - \beta_{k\!+\!1}\Ab\hat{\ub}^{\ast}_{\gamma_{k\!+\!1}}(\xhat^k)\right).
\end{equation*}
\State\label{step:4}Update  $\xhat^{k\!+\!1}  := \xb^{k\!+\!1} + \left(\frac{k+\bar{c} - 1}{k + \bar{c} + 1}\right)(\xb^{k\!+\!1} -  \xb^k)$.
\State\label{step:5}Compute $\gamma_{k+2} := \frac{\bar{c}\gamma_1}{k+\bar{c}+1}$.
\Statex{\hskip-4ex}\textbf{End~for}
\end{algorithmic}
\end{normalsize}
\end{algorithm}

The following theorem proves the convergence of Algorithm \ref{alg:A1} and estimates its worst-case iteration-complexity.

\vspace{-0.75ex}
\begin{theorem}\label{th:convergence_v2}
Let $\{\xb^k\}$ be the sequence generated by Algorithm~\ref{alg:A1} using $\bar{c} = 1$.
Then, for $k\geq 1$, we have
\begin{equation}\label{eq:convergence_rate_v2a}
F(\xb^{k}) - \Fopt \leq  \frac{\norm{\Ab}^2\Vert\xb^0 \!\!-\! \xopt\Vert^2}{2\gamma_1k}   + \frac{3\gamma_1D_{\Uc}}{k} + \frac{\gamma_1(L_b-1)\left(\ln(k) + 1\right) D_{\Uc}}{k}.
\end{equation}
If $b_{\Uc}$ is chosen so that $L_b = 1$ $($e.g., $b_{\Uc}(\cdot) := \frac{1}{2}\Vert\cdot-\bar{\ub}^c\Vert^2$$)$, then \eqref{eq:convergence_rate_v2a} reduces to
\vspace{-0.5ex}
\begin{align}\label{eq:convergence_rate2_v2}
F(\xb^{k}) - \Fopt &\leq  \frac{\norm{\Ab}^2\Vert\xb^0 - \xopt\Vert^2}{2\gamma_1k} + \frac{3\gamma_1D_{\Uc}}{k}, ~~(\forall k\geq 1).
\vspace{-0.5ex}
\end{align}
Consequently, if we set  $\gamma_1 := \frac{R_0\norm{\Ab}}{\sqrt{6D_{\Uc}}}$, which is independent of $k$, then
\begin{align}\label{eq:worst_case2_v2}
F(\xb^{k}) - \Fopt  \leq  \frac{R_0\norm{\Ab}\sqrt{6D_{\Uc}}}{k}~~~(\forall k\geq 1),
\end{align}
where $R_0 := \norm{\xb^0 - \xopt}$.

In this case, the worst-case iteration-complexity of Algorithm \ref{alg:A1} to achieve an $\varepsilon$-solution $\xb^{k}$ to \eqref{eq:com_min} such that $F(\xb^{k})-\Fopt\leq\varepsilon$ is $k_{\max} := \mathcal{O}\Big(\frac{R_0\norm{\Ab}\sqrt{D_{\Uc}}}{\varepsilon}\Big)$.
\end{theorem}

\begin{proof}
From \eqref{eq:param_update_rules}, $\bar{c} = 1$  we have $\frac{\tau_{k-1}^2}{\gamma_k} = \frac{(k+\bar{c}-1)}{\bar{c}\gamma_1(k+\bar{c}-1)^2}  = \frac{1}{\gamma_1k}$. Using this bound and $S_{k-1} \leq \gamma_1^2(L_b-1)\left[\ln(k) + 1\right] + 2\gamma_1^2$ into \eqref{eq:key_estimate2_new} we get
\begin{align*}
F_{\gamma_k}(\xb^k) - \Fopt &\leq \frac{1}{\gamma_1k}\left[\frac{\norm{\Ab}^2}{2}\Vert\xb^0 - \xopt\Vert^2 +\frac{\gamma_1(1-\tau_0)}{\tau_0^2}[F_{\gamma_0}(\xb^0) - \Fopt]\right]\\
&+ \frac{\left(\gamma_1(L_b-1)\left[\ln(k) + 1\right] + 2\gamma_1\right)D_{\Uc}}{k}.
\end{align*}
Since $F(\xb^k) - F_{\gamma_k}(\xb^k) \leq \gamma_kD_{\Uc}$ due to \eqref{eq:key_est1}, and $\gamma_k = \frac{\gamma_1\bar{c}}{k+\bar{c}-1} = \frac{\gamma_1}{k}$. Substituting this inequality into the last estimate, and using $\tau_0 = \frac{1}{\bar{c}} = 1$, we obtain \eqref{eq:convergence_rate_v2a}.

If we choose $b_{\Uc}$ such that $L_b = 1$, e.g., $b_{\Uc}(\cdot) := (1/2)\Vert\cdot-\ubar^c\Vert^2$, then  $S_k \leq 2\gamma_1^2$ as shown in \eqref{eq:Sk_bound}. 
Using this, it follows from \eqref{eq:convergence_rate_v2a} that $F(\xb^k) - \Fopt \leq \frac{\norm{\Ab}}{2\gamma_1k}R_0^2 + \frac{3\gamma_1}{k\!+\!1}D_{\Uc}$. 
By minimizing the right hand side of this estimate w.r.t $\gamma_1 > 0$, we have $\gamma_1 := \frac{R_0\norm{\Ab}}{\sqrt{6D_{\Uc}}}$ and hence, $F(\xb^{k}) \!-\! \Fopt  \leq  \frac{R_0\norm{\Ab}\sqrt{6D_{\Uc}}}{k}$, which is exactly \eqref{eq:worst_case2_v2}.
The last statement is a direct consequence of \eqref{eq:worst_case2_v2}.
\Eproof
\end{proof}

For general prox-function $b_{\Uc}$ with $L_b > 1$, Theorem \ref{th:convergence_v2} shows that the convergence rate of Algorithm \ref{alg:A1} is $\mathcal{O}\left(\frac{\ln(k)}{k}\right)$, which is similar to \cite{boct2012variable}. However, when $L_b$ is close to $1$,  the last term in \eqref{eq:convergence_rate_v2a} is better than \cite[Theorem 1]{boct2012variable}.

\vspace{-1ex}
\begin{remark}\label{re:compare}
Let $b_{\Uc}(\cdot) := (1/2)\Vert\cdot-\ubar^c\Vert^2$. Then, \eqref{eq:worst_case2_v2} shows that the number of maximum iterations in Algorithm \ref{alg:A1} is $k_{\max} :=  \left\lfloor\frac{R_0\Vert A\Vert\sqrt{6D_{\Uc}}}{\varepsilon}\right\rfloor$, which is the same, $k_{\max} := \left\lfloor \frac{2\sqrt{2}\norm{\Ab}R_0\sqrt{D_{\Uc}}}{\varepsilon}\right\rfloor$,  as in \eqref{eq:nes_bound} (with different factors, $\sqrt{6}$ and $2\sqrt{2}$).
\end{remark}

\vspace{-5ex}
\section{Exploiting structures for special cases}\label{sec:exploit_struct}
\vspace{-2ex}
For general smooth proximity function $b_{\Uc}$ with $L_b > 1$, we can achieve the $\mathcal{O}\left(\frac{(L_b - 1)\ln(k)}{k}\right)$ convergence rate. When $L_b = 1$, we obtain exactly the $\mathcal{O}\left(\frac{1}{k}\right)$ rate as in \cite{Nesterov2005c}.
In this section, we consider three special cases of \eqref{eq:com_min} where we use the quadratic proximity function $b_{\Uc}(\cdot) := (1/2)\norm{\cdot -  \bar{\ub}^c}^2$.
Then, we specify Algorithm~\ref{alg:A1} for the $L_g$-smooth objective function $g$ in \eqref{eq:com_min}.

\vspace{-3.5ex}
\subsection{Fenchel conjugate}
\vspace{-2.5ex}
Let $f^{*}$ be the Fenchel conjugate of $f$.
We can write $f$ in the form of \eqref{eq:fx_structure} as
\begin{equation*}
f(\xb) = \max_{\ub}\set{ \iprods{\xb,\ub} - f^{\ast}(\ub) : \ub\in\dom{f^{\ast}} }.
\end{equation*}
We can smooth $f$ by using $b_{\Uc}(\ub) := (1/2)\norm{\ub}_2^2$ as
\begin{equation*}
f_{\gamma}(\xb) := \max_{ \ub\in\dom{f^{\ast}} }\set{ \iprods{\xb,\ub} - f^{\ast}(\ub) - (\gamma/2)\norm{\ub}_2^2  }  = \frac{\norm{\xb}^2}{2\gamma} -{}^{\gamma^{-1}}f^{\ast}(\gamma^{-1}\xb),
\end{equation*}
where ${}^\beta{h}$ is the Moreau envelope of a convex function $h$ with a parameter $\beta$ \cite{Bauschke2011}.
In this case, $\ub^{\ast}_{\gamma}(\xb) = \prox_{\gamma^{-1}f^{\ast}}( \gamma^{-1}\xb) = \gamma^{-1}(\xb - \prox_{\gamma f}(\xb))$. 
Hence, $\nabla{f_{\gamma}}(\xb) = \gamma^{-1}(\xb - \prox_{\gamma f}(\xb))$. 
The main step, Step \ref{step:3}, of Algorithm \ref{alg:A1} becomes
\begin{equation*}
\xb^{k\!+\!1} = \prox_{\gamma_{k\!+\!1}g}\big( \prox_{\gamma_{k\!+\!1}f}(\hat{\xb}^k) \big).
\end{equation*}
Hence, Algorithm \ref{alg:A1} can be applied to solve \eqref{eq:com_min} using the proximal operator of $f$ and $g$.
The worst-case complexity bound in Theorem \ref{th:convergence_v2} becomes $\mathcal{O}\left( \frac{D_{\dom{f^{\ast}}}R_0}{\varepsilon}\right)$, where $D_{\dom{f^{\ast}}} := \displaystyle\max_{\ub\in\dom{f^{\ast}}} \norm{\ub}$ is the diameter of $\dom{f^{\ast}}$.

\vspace{-3.5ex}
\subsection{Composite convex minimization with linear operator}
\vspace{-2.5ex}
We consider the following composite convex  problem with a linear operator that covers many important applications in practice, see, e.g.,~\cite{argyriou2014hybrid,Bauschke2011,Combettes2011a}:
\begin{equation}\label{eq:composite_cvx}
F^{\star} := \min_{\xb\in\R^p}\set{ F(\xb) := f(\Ab\xb) + g(\xb)},
\end{equation}
where $f$ and $g$ are two proper, closed and convex functions, and $\Ab$ is a linear operator from $\R^p$ to $\R^n$.

We first write $f(\Ab\xb) := \max_{\ub}\set{\iprods{\Ab\xb, \ub} - f^{\ast}(\ub) : \ub\in\dom{f^{\ast}}}$. 
Next, we choose a quadratic smoothing proximity function $b_{\Uc}(\ub) := (1/2)\norm{\ub - \bar{\ub}^c}^2$ for fixed $\bar{\ub}^c\in\dom{f^{\ast}}$, and define  $\Uc := \dom{f^{\ast}}$.
Using this smoothing prox-function, we obtain a smoothed approximation of $f(\Ab\xb)$ as follows:
\begin{equation*}
f_{\gamma}(\Ab\xb) := \max_{\ub}\set{\iprods{\Ab\xb, \ub} - f^{\ast}(\ub) - (\gamma/2)\norm{\ub - \bar{\ub}^c}^2 : \ub\in\dom{f^{\ast}}}.
\end{equation*}
In this case, we can compute $\ub^{\ast}_{\gamma}(\Ab\xb) = \prox_{\gamma^{-1}f^{\ast}}\left(\ubar^c + \gamma^{-1}\Ab\xb\right)$ by using the proximal operator of $f^{\ast}$.
By Fenchel-Moreau's decomposition $\prox_{\gamma^{-1}f^{\ast}}(\gamma^{-1}\vb) = \gamma^{-1}(\vb -  \prox_{\gamma f}(\gamma \vb))$ as above, we can compute $\prox_{\gamma^{-1}f^{\ast}}$ using the proximal operator of $f$.
In this case, we can specify the proximal-gradient step~\eqref{eq:prox_grad_step} as
\begin{equation*}
\left\{\begin{array}{ll}
\hat{\ub}^{\ast}_k &:= \prox_{\gamma_{k+1}^{-1}f^{\ast}}\left(\ubar^c + \gamma_{k+1}^{-1}\Ab\hat{\xb}^k\right) \vspace{0.5ex}\\
& = \ubar^c + \gamma_{k+1}^{-1}\left(\Ab\hat{\xb}^k - \prox_{\gamma_{k+1}f}\big(\gamma_{k+1}\ubar^c + \Ab\hat{\xb}^k\big)\right)\vspace{0.75ex}\\
\xb^{k+1} &:= \prox_{\beta_{k+1} g}\left(\hat{\xb}^k - \beta_{k+1}\Ab^{\top}\hat{\ub}^{\ast}_k\right),
\end{array}\right.
\end{equation*}
where $\beta_{k+1} := \gamma_{k+1}\norm{\Ab}^{-2}$.
Using this proximal gradient step in Algorithm~\ref{alg:A1}, we still obtain the complexity as in Theorem~\ref{th:convergence_v2}, which is $\mathcal{O}\Big(\frac{\norm{\xb^0 - \xopt}\norm{\Ab}\sqrt{D_{\Uc}}}{\varepsilon}\Big)$, where the domain $\Uc := \dom{f^{\ast}}$ of $f^{\ast}$ is assumed to be bounded.

\vspace{-3.5ex}
\subsection{The decomposable structure}\label{subsec:decomposable}
\vspace{-2.5ex}
The function $\varphi$ and the set $\Uc$ in \eqref{eq:fx_structure} are said to be \textit{decomposable} if they can be represented as follows:
\vspace{-0.5ex}
\begin{equation}\label{eq:decom_structure}
\varphi(\ub) := \sum_{i=1}^m\varphi_i(\ub_i), ~~~\text{and}~~~ \Uc := \Uc_1\times \cdots\times \Uc_m,
\vspace{-0.5ex}
\end{equation}
where $m\geq 2$, $\ub_i \in \R^{n_i}$, $\Uc_i\subseteq\R^{n_i}$ and $\sum_{i=1}^mn_i = n$.
In this case, we also say that problem \eqref{eq:com_min} is \textit{decomposable}.

The structure \eqref{eq:decom_structure} naturally arises in linear programming and monotropic programming.
In addition, many nondecomposable problems such as consensus optimization, empirical loss optimization, conic programming and geometric programming can also be reformulated into \eqref{eq:com_min} with the structure \eqref{eq:decom_structure}. 
The decomposable structure \eqref{eq:decom_structure} immediately supports  parallel and distributed computation.
Exploiting this structure, one can design new parallel and distributed optimization algorithms using the same approach as in  Algorithm \ref{alg:A1} for solving \eqref{eq:com_min}, see, e.g., \cite{Bertsekas1989b,Boyd2011,Rockafellar1985,TranDinh2012g}. 

Under the structure \eqref{eq:decom_structure},  we choose a decomposable smoothing function $b_{\Uc}(\ub) := \sum_{i=1}^mb_{\Uc_i}(\ub_i)$, where $b_{\Uc_i}$ is the prox-function of $\Uc_i$ for $i=1,\cdots,m$.
The smoothed function $f_{\gamma}$ for $f$ is decomposable, and is represented as follows:
\vspace{-0.5ex}
\begin{equation}\label{eq:decom_subprob} 
f_{\gamma}(\xb) := \sum_{i=1}^m\set{ f_{\gamma}^i(\xb) := \displaystyle\max_{\ub_i\in\Uc_i}\set{\iprods{\xb, \Ab_i\ub_i} - \varphi_i(\ub_i) - \gamma b_{\Uc_i}(\ub_i)} }.
\vspace{-0.5ex}
\end{equation}
Let us denote by $\uast_{\gamma,i}(\Ab_i^{\top}\xb)$ the unique solution of the subproblem $i$ in \eqref{eq:decom_subprob} for $i=1,\cdots, m$.
Then, under the decomposable structure, the evaluation of $f_{\gamma}$ and $\uast_{\gamma}(\Ab^{\top}\xb) := [\uast_{\gamma,1}(\Ab_1^{\top}\xb),\cdots, \uast_{\gamma,m}(\Ab_m^{\top}\xb)]$ can be computed in parallel.

If we apply Algorithm \ref{alg:A1} to solve \eqref{eq:com_min} with the structure \eqref{eq:decom_structure}, then we have the following guarantee on the objective residual:
\begin{equation*} 
\begin{array}{ll}
F(\xb^{k}) - \Fopt &\leq   \frac{L_{\Ab}\Vert\xb^0 \!\!-\! \xopt\Vert^2}{2\gamma_1k}   + \frac{\gamma_1D_{\Uc}}{k}\left(3 + (L_b -1)\left(\ln(k+1) + 1\right)\right),
\end{array}
\end{equation*}
where $L_{\Ab} := \sum_{i=1}^m\norm{\Ab_i}^2$, $L_b := \max\set{ L_{b_i} : 1\leq i\leq m}$ and $D_{\Uc} := \sum_{i=1}^mD_{\Uc_i}$.
Hence, the convergence rate of Algorithm \ref{alg:A1}  stated in Theorem \ref{th:convergence_v2} is $\mathcal{O}\left(\frac{\ln(k)}{k}\right)$.
If we choose $b_{\Uc_i}(\cdot) := (1/2)\Vert\cdot-\ubar^c_i\Vert^2$ for all $i=1,\cdots, m$, then $L_b = 1$.
Consequently, we obtain the $\mathcal{O}\left(\frac{L_{\Ab}R_0\sqrt{D_{\Uc}}}{\varepsilon}\right)$-worst-case iteration-complexity.

\vspace{-3.5ex}
\subsection{The Lipschitz gradient structure}
\vspace{-2.5ex}
If $g$ is smooth and its gradient  $\nabla{g}$ is Lipschitz continuous with the Lipschitz constant $L_g > 0$, then $F_{\gamma} := f_{\gamma} + g\in\Fc_{L}^{1,1}$, i.e., $\nabla{F_{\gamma}}$ is Lipschitz continuous with the Lipschitz constant $L_{F_{\gamma}} :=  L_g + \gamma^{-1}\norm{\Ab}^2$.

We replace the proximal-gradient step \eqref{eq:prox_grad_step} using in Algorithm~\ref{alg:A1} by the following ``full'' gradient step
\vspace{-0.5ex}
\begin{equation}\label{eq:grad_step}
\xb^{k\!+\!1} := \xhat^k  -\beta_{k\!+\!1}\big(\nabla{g}(\xhat^k) + \Ab\ub^{\ast}_{\gamma_{k\!+\!1}}(\Ab^{\top}\xhat^k)\big),
\vspace{-0.5ex}
\end{equation}
where $\uast_{\gamma_{k\!+\!1}}(\Ab^{\top}\xhat^k)$ is computed by \eqref{eq:uast_beta} and $\beta_{k\!+\!1} :=  \frac{1}{L_g + \gamma_{k\!+\!1}^{-1}\norm{\Ab}^2}$ is a given step-size.
Unlike \eqref{eq:param_update_rules}, we update  the parameters $\tau_k$ and $\gamma_k$ as
\vspace{-0.5ex}
\begin{equation*}
\tau_k := \frac{1}{k+1}~~~\text{and}~~~~\gamma_{k+1} := \frac{k\gamma_k\norm{\Ab}^2}{L_g\gamma_k + \norm{\Ab}^2(k+1)},
\vspace{-0.5ex}
\end{equation*}
where $\gamma_1 := \frac{\norm{\Ab}^2}{L_g}$ is fixed.
We name this variant  as Algorithm \ref{alg:A1}(b).

The following corollary summarizes the convergence properties of this variant, whose proof can be found in Appendix~\ref{apdx:co:update_param2}.

\begin{corollary}\label{co:update_param2}
Assume that $g\in\Fc_L^{1,1}$ with the Lipschitz constant $L_g \geq 0$.
Let $\{ \xb^k\}$ be the sequence generated by Algorithm $\mathrm{\ref{alg:A1}(b)}$. Then, for $k\geq 1$, one has
\vspace{-0.5ex}
\begin{equation}\label{eq:convergence5}
{\!\!\!\!\!}\begin{array}{ll}
{\!\!\!\!}F(\xb^k) \!-\! F^{\star} \!\leq\! \frac{3L_g}{2k}\norm{\xb^0 \!-\! \xopt}^2 + \frac{\norm{\Ab}^2}{L_gk}\left(\frac{2L_b}{L_g} \!+\! 1\right)D_{\Uc} 
\! +\!  \frac{(L_b \!-\! 1)\norm{\Ab}^2}{L_g^2k}\left(\ln(k) \!+\! 1\right)D_{\Uc}.
\end{array}{\!\!\!\!\!}
\vspace{-0.5ex}
\end{equation}
If we choose $b_{\Uc}$ such that $L_b = 1$, then \eqref{eq:convergence5} reduces to 
\vspace{-0.5ex}
\begin{equation*}
F(\xb^k) - F^{\star} \leq \frac{3L_g}{2k}\norm{\xb^0 - \xopt}^2 + \frac{\norm{\Ab}^2}{L_g^2k}(L_g + 2)D_{\Uc}.
\vspace{-0.5ex}
\end{equation*}
Consequently, the worst-case iteration-complexity of Algorithm $\mathrm{\ref{alg:A1}(b)}$ is $\mathcal{O}\Big(\frac{1}{\varepsilon}\Big)$.
\end{corollary}

\vspace{-3.5ex}
\section{Application to general constrained convex optimization}\label{sec:app}
\vspace{-2ex}
In this section, we customize Algorithm~\ref{alg:A1}  to solve the following general constrained convex optimization problem:
\begin{equation}\label{eq:constr_cvx}
\varphi^{\star} := \min_{\ub\in\R^n }\Big\{ \varphi(\ub) : \Ab\ub - \bb \in -\Kc, ~\ub\in\Uc \Big\},
\end{equation}
where $\varphi$ is a proper, closed and convex function from $\R^n\to\Rext$, $\Ab\in\R^{p\times n}$, $\bb\in\R^p$, $\Uc$ and $\Kc$ are two nonempty, closed and convex set in $\R^n$ and $\R^p$, respectively.
Without loss of generality, we can assume that $\varphi$ and $\Uc$ are decomposable as in \eqref{eq:decom_structure} with $m\geq 1$.

Associated with the primal setting \eqref{eq:constr_cvx}, we consider its dual problem
\vspace{-0.5ex}
\begin{equation}\label{eq:constr_cvx_dual}
F^{\star} := \min_{\xb\in\R^p}\Big\{ F(\xb) := \max_{\ub\in\Uc}\set{ \iprods{\xb, \Ab\ub} - \varphi(\ub)}  - \iprods{\bb,\xb} + \max_{\rb\in\Kc}\iprods{\xb,\rb} \Big\}.
\vspace{-0.5ex}
\end{equation}
Clearly,  \eqref{eq:constr_cvx_dual} has the same form as \eqref{eq:com_min} with $f(\xb) := \displaystyle\max_{\ub}\set{ \iprods{\xb,\Ab\ub} - \varphi(\ub): \ub\in\Uc }$ and $g(\xb) := s_{\Kc}(\xb) - \iprods{\bb,\xb}$, where $s_{\Kc}$ is the support function of $\Kc$.

We now specify Algorithm~\ref{alg:A1} to solve this dual problem.
Computing $\ub_{\gamma}^{*}(\xb)$ requires to solve the following sub-problem:
\begin{equation*}
\ub_{\gamma}^{*}(\xb) := \argmin_{\ub}\set{\iprods{\xb, \Ab\ub} - \varphi(\ub) - \gamma b_{\Uc}(\ub) }.
\end{equation*} 
The proximal-step of $g$ becomes $\prox_g(\xb) := \prox_{s_{\Kc}}(\xb+\bb) = (\xb+\bb) - \mathrm{proj}_{\Kc}(\xb+\bb)$, where $\mathrm{proj}_{\Kc}(\cdot)$ is the projection onto $\Kc$.
Together with the dual steps, we use an adaptive weighted averaging scheme
\begin{equation}\label{eq:averaging_scheme}
\bar{\ub}^k := \Gamma_k^{-1}\sum_{i=0}^k\tau_i^{-1}\gamma_{i+1}\uast_{\gamma_{i+1}}(\xhat^i),~~~\text{and}~~~\Gamma_k :=  \sum_{i=0}^k\tau_i^{-1}\gamma_{i+1},
\end{equation}
 to construct an approximate primal solution $\bar{\ub}^k$ to an optimal solution $\ub^{\star}$ of \eqref{eq:constr_cvx}.
Clearly, we can compute $\bar{\ub}^k$ recursively starting from $\bar{\ub}^0 := \boldsymbol{0}^n$ as
\begin{equation}\label{eq:averaging_scheme2}
\bar{\ub}^k := (1-\nu_k)\bar{\ub}^{k-1} + \nu_k\uast_{\gamma_{k\!+\!1}}(\xhat^k), ~~\text{where}~~\nu_k := (\Gamma_k\tau_k)^{-1}\gamma_{k\!+\!1} \in (0, 1].
\end{equation}
We incorporate this scheme into Algorithm~\ref{alg:A1} to solve \eqref{eq:constr_cvx}. 
While Algorithm \ref{alg:A1} constructs an approximate solution to the dual problem \eqref{eq:constr_cvx_dual}, \eqref{eq:averaging_scheme2} allows us to recover an approximate  solution $\bar{\ub}^k$ of the primal problem \eqref{eq:constr_cvx}.
We name this algorithmic variant as Algorithm \ref{alg:A1}(c).

We  specify the convergence guarantee of Algorithm \ref{alg:A1}(c) in the following theorem.
The proof of this theorem is given in Appendix~\ref{apdx:th:primal_recovery}.

\vspace{-1ex}
\begin{theorem}\label{th:primal_recovery}
Assume that $b_{\Uc}$ is chosen such that $L_b = 1$, and $\bar{c} = 1$ in \eqref{eq:param_update_rules}. Let $\{(\xb^k,\bar{\ub}^k)\}$ be generated by Algorithm $\mathrm{\ref{alg:A1}(c)}$.
Then $\{\bar{\ub}^k\}\subset\Uc$ and
\begin{align}\label{eq:primal_recovery2}
{\!\!\!\!\!\!}\left\{\begin{array}{l}
-\Vert\xopt\Vert\dist{\bb - \Ab\bar{\ub}^k, \Kc} \leq \varphi(\ubar^k) - \varphi^{\star} \leq \frac{ \norm{\Ab}^2\norm{\xb^0}^2 +2 (\gamma_1 + 2\gamma_1^2)D_{\Uc} }{\gamma_1(k+1)},\vspace{1.5ex}\\
\dist{\bb \!-\! \Ab\bar{\ub}^k, \Kc} \leq \frac{\norm{\Ab}^2\Big(\norm{\xb^0 \!-\! \xopt} + \sqrt{\norm{\xb^0 -  \xopt}^2 +   2\norm{\Ab}^{-2}(2\gamma_1^2 + \gamma_1)D_{\Uc}}\Big)}{\gamma_1(k\!+\!1)}.
\end{array}\right.{\!\!\!\!}
\end{align}
Consequently, the worst-case iteration-complexity of  Algorithm $\mathrm{\ref{alg:A1}(c)}$ to achieve an $\varepsilon$-solution $\bar{\ub}^k$ such that $\vert\varphi(\bar{\ub}^k) - \varphi^{\star}\vert \leq \varepsilon$ and $\dist{\bb-\Ab\bar{\ub}^k, \Kc} \leq \varepsilon$ is $\mathcal{O}\left(\frac{1}{\varepsilon}\right)$.
\end{theorem}
Theorem \ref{th:primal_recovery} shows that  Algorithm \ref{alg:A1}(c) has the $\mathcal{O}\left(\frac{1}{\varepsilon}\right)$ worst-case iteration-complexity on the primal objective residual and feasibility violation for \eqref{eq:constr_cvx}.

\vspace{-3.5ex}
\section{Preliminarily numerical experiments}\label{sec:num_exp}
\vspace{-2ex}
We demonstrate the performance of   Algorithm~\ref{alg:A1}  for solving the three well-known convex optimization problems.
The first example is a LASSO problem with $\ell_1$-loss \cite{Yang2011}, the second one is a square-root LASSO studied in  \cite{Belloni2011}, 
and the last example is an image deblurring problem with a non-smooth data fidelity function (e.g., the $\ell_1$-norm or the $\ell_2$-norm function).

\vspace{-3.5ex}
\subsection{The $\ell_1\text{-}\ell_1$-regularized LASSO}\label{subsec:robustLasso}
\vspace{-2.0ex}
We consider the $\ell_1\text{-}\ell_1$-regularized LASSO problem studied in \cite{Yang2011} as follows:
\begin{equation}\label{eq:l1_lasso}
\Fopt := \min\set{ F(\xb) := \norm{\Bb\xb - \bb}_1 + \lambda \norm{\xb}_1 : \xb\in\R^p},
\end{equation}
where  $\Bb$ and $\bb$ are defined as in \eqref{eq:sqrt_lasso0}, and $\lambda > 0$ is a regularization parameter.

The function $f(\xb) := \norm{\Bb\xb - \bb}_1 =  \max\set{ \iprods{\Bb^{\top}\ub, \xb} - \iprods{\bb,\ub} : \norm{\ub}_{\infty} \leq 1}$ falls into the decomposable case considered in Subsection \ref{subsec:decomposable}. Hence, we can smooth $f$ using the quadratic prox-function to obtain
\vspace{-0.75ex}
\begin{equation*}
f_{\gamma}(\xb) := \max_{\ub}\Big\{ \iprods{\xb, \Bb^{\top}\ub} - \iprods{\bb,\ub} - (\gamma/2)\norm{\ub}^2 : \ub\in\Bc_{\infty} \Big\}.
\vspace{-0.75ex}
\end{equation*}
Clearly, we can  show that $\uast_{\gamma}(\Bb\xb) :=  \mathrm{proj}_{\Bc_{\infty}}\left(\gamma^{-1}(\Bb\xb - \bb)\right)$.
In this case, we also have $D_{\Bc_{\infty}} := \frac{1}{2}n$ and $\Uc := \Bc_{\infty}$. 

Now, we apply Algorithm~\ref{alg:A1} to solve problem \eqref{eq:sqrt_lasso0}. 
To verify the theoretical bound in Theorem \ref{th:convergence_v2}, we use CVX \cite{Grant2006} to solve \eqref{eq:sqrt_lasso0}  and obtain a high accuracy approximate solution $\xopt$. Then, we can compute $R_0 := \norm{\xb^0 - \xopt}_2$, and choose $\gamma_1 \equiv \gamma_1^{*} := \frac{\norm{\Bb}R_0}{\sqrt{6D_{\Bc_{\infty}}}}$.
From Theorem \ref{th:convergence_v2}, we have $F(\xb^k) \!-\! \Fopt \leq \frac{R_0\norm{\Bb}\sqrt{6D_{\Bc_{\infty}}}}{k}$, which is the worst-case bound of Algorithm \ref{alg:A1}, where $k$ is the iteration number.

For our comparison, we also implement the smoothing algorithm in \cite{Nesterov2005c} using the quadratic prox-function.
As indicated in \eqref{eq:nes_bound}, we set $\gamma \!\equiv\! \gamma^{*} := \frac{\sqrt{2}\norm{B}R_0}{\sqrt{D_{\Uc}}(k+1)}$.
Hence, we also obtain the theoretical upper bound $F(\xb^k) \!-\! \Fopt \!\leq\! \frac{2\sqrt{2}\norm{B}R_0\sqrt{D_{\Uc}}}{(k+1)}\!$.
We name this algorithm as \textrm{Non-adapt. Alg.} (non-adaptive algorithm).

\begin{figure*}[ht!]
\begin{center}
\vspace{-4ex}
\includegraphics[width =1\textwidth]{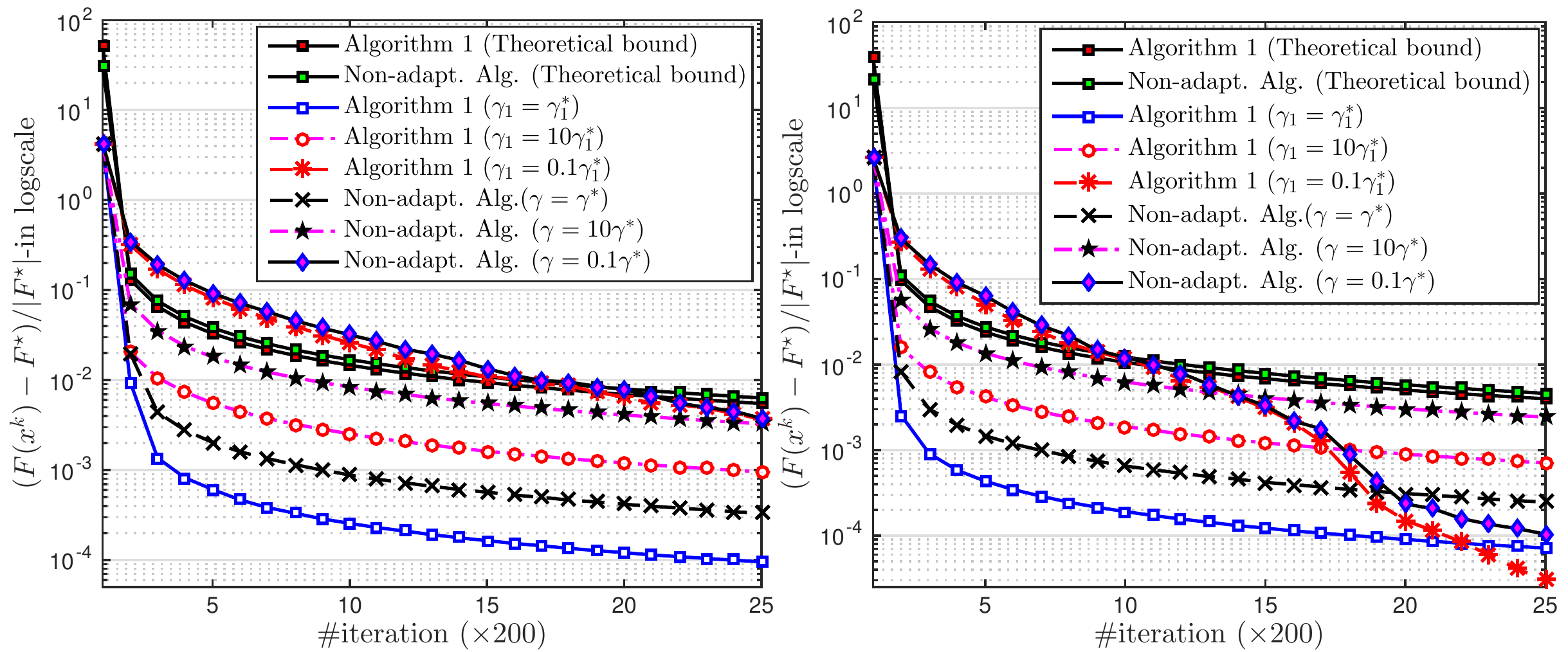}
\vspace{-3ex}
\caption{The empirical performance vs. the theoretical bounds of the $6$ algorithmic variants (Left: non-correlated data, Right: correlated data).}\label{fig:robust_exam1000_evst}
\vspace{-6ex}
\end{center}
\end{figure*} 

The test data is generated as follows:
Matrix $\Bb \in \R^{n\times p}$ is generated randomly using the standard Gaussian distribution $\mathcal{N}(0, 1)$. We consider two cases. In the first case, we use non-correlated data, while in the second case, we generate $\Bb$ with $50\%$ correlated columns as $\Bb(:,j+1) = 0.5\Bb(:,j) + \textrm{randn}(:)$. 
The observed measurement vector $\bb$ is generated as $\bb := \Bb\xb^{\natural} + \mathcal{N}(0, 0.05)$, where $\xb^{\natural}$ is a given $s$-sparse vector generated randomly using $\mathcal{N}(0, 1)$.

We test both algorithms: Algorithm~\ref{alg:A1} and \textrm{Non-adapt.~Alg.} on  two problem instances of the size $(p, n, s) = (1000, 350, 100)$ (with and without correlated data, respectively). 
We sweep along the values of $\lambda$ to find an optimal value for $\lambda$ which are $\lambda = 6.2105$ for non-correlated data, and $\lambda = 5.7368$ for correlated data, respectively. 
For comparison, we first select the optimal value for $\gamma_1 := \gamma_1^{*}$ and $\gamma := \gamma^{*}$ in both algorithms. 
Then, we consider two cases: $(i)$ $\gamma_1 := 10\gamma_1^{*}$ and $\gamma := 10\gamma^{*}$, and $(ii)$~$\gamma_1 := 0.1\gamma_1^{*}$ and $\gamma := 0.1\gamma^{*}$.

Figure \ref{fig:robust_exam1000_evst} plots the empirical bounds of the $6$ variants vs. the theoretical bounds from $200$ to $10,000$ iterations.
Obviously, both algorithms show their empirical rate which is much better than their theoretical bound. But if we change the smoothness parameters, the guarantee is no longer preserved.
Algorithm~\ref{alg:A1} shows a better performance than \textrm{Non-adapt.~Alg.}  in both cases. 

\vspace{-3.5ex}
\subsection{Square-root LASSO}\label{subsec:sqrtLasso}
\vspace{-2ex}
We consider the following well-known square-root LASSO problem:
\vspace{-0.5ex}
\begin{equation}\label{eq:sqrt_lasso0}
\min_{\xb\in\R^p}\Big\{ F(\xb) := \norm{\Bb\xb - \bb}_2 + \lambda\norm{\xb}_1 \Big\}.
\vspace{-0.5ex}
\end{equation}
As proved in \cite{Belloni2011}, if matrix $\Bb$ is Gaussian, then we can select the regularization parameter $\lambda$ such that we can obtain exact recovery to the true solution $\xb^{\natural}$.

The function $f$ defined by $f(\xb) := \norm{\Bb\xb - \bb}_2$ can be written as
\vspace{-0.5ex}
\begin{equation*}
f(\xb) =  \max_{\ub}\set{ \iprods{\Bb^{\top}\ub, \xb} - \iprods{\bb,\ub} : \norm{\ub}_2 \leq 1}.
\vspace{-0.5ex}
\end{equation*}
Let $\Bc_2 := \set{\ub\in\R^n : \norm{\ub}_2 \leq 1}$ be the $\ell_2$-norm ball. We choose $b(\ub) := \frac{1}{2}\norm{\ub}_2^2$ as a prox-function for $\Bc_2$. 
Then, we can smooth $f$ using $b(\cdot) := \frac{1}{2}\norm{\cdot}_2^2$  as
\begin{equation*}
\vspace{-0.5ex}
f_{\gamma}(\xb) := \max_{\ub}\set{ \iprods{\xb, \Bb^{\top}\ub} - \iprods{\bb,\ub} - (\gamma/2)\norm{\ub}_2^2 : \ub\in\Bc_2}.
\vspace{-0.5ex}
\end{equation*}
Clearly, $\uast_{\gamma}(\xb) :=  \mathrm{proj}_{\Bc_2}\left(\gamma^{-1}(\Bb\xb - \bb)\right)$ is the solution of the maximization problem,  where $\mathrm{proj}_{\Bc_2}$ is the projection onto $\Bc_2$.
Moreover, we have $D_{\Uc} = \frac{1}{2}$. 

Now, we apply Algorithm~\ref{alg:A1} to solve problem \eqref{eq:sqrt_lasso0}. 
We choose $\bar{c} := 1$ and set $\gamma_1 \equiv \gamma_1^{*} := \frac{\norm{\Ab}R_0}{\sqrt{6D_{\Uc}}}$, where $R_0 := \norm{\xb^0 - \xopt}_2$. 
We also estimate the theoretical upper bound indicated in Theorem \ref{th:convergence_v2} for $F(\xb^k) - \Fopt$ using \eqref{eq:convergence_rate_v2a}, which is $\frac{\Vert\Ab\Vert R_0\sqrt{6D_{\Uc}}}{k}$.
We implement the smoothing algorithm in \cite{Nesterov2005c} for our comparison by using the same prox-function.
The parameter of this algorithm is set as in the previous example. 

The data test is generated as in Subsection \ref{subsec:robustLasso}.
We also perform the test on two problem instances of size $(p, n, s) = (1000, 350, 100)$: non-correlated data and correlated data.
We choose the regularization parameter $\lambda$ as  suggested in \cite{Belloni2011}.
We use the same setting for the smoothness parameter $\gamma$ in both algorithms as in Subsection \ref{subsec:robustLasso}.
In this case, the theoretical upper bound of Algorithm~\ref{alg:A1} depends on the log-term which is scaled by the condition number of $\Bb\Bb^{\top}$, and is worse than in \textrm{Non-adapt.~Alg.} variant.

\begin{figure*}[ht!]
\begin{center}
\vspace{-4ex}
\includegraphics[width = 1.01\textwidth]{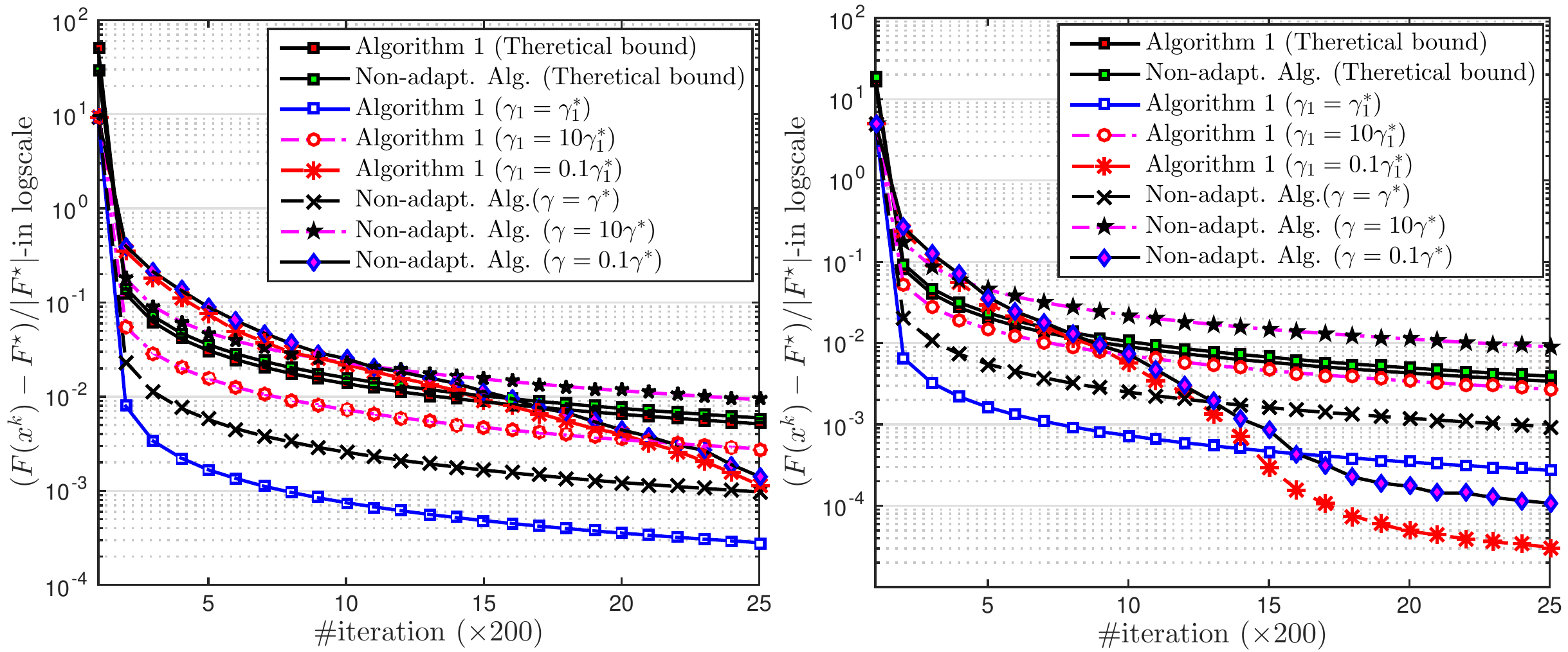}
\vspace{-3ex}
\caption{\footnotesize The empirical performance vs. the theoretical bounds of the $6$ algorithmic variants (Left: non-correlated data, Right: $50\%$-correlated collumns).}\label{fig:sqrt_exam1000_evst}
\end{center}
\vspace{-6ex}
\end{figure*}

Figure \ref{fig:sqrt_exam1000_evst} plots the empirical bounds of the $6$ variants vs. the theoretical bounds from $200$ to $10,000$ iterations.
Obviously, both algorithms show their empirical rate which is much better than their theoretical bound. 
Algorithm~\ref{alg:A1} gives a better performance than the nonadaptive method in this example. 
We  note that the theoretical bound in Algorithm~\ref{alg:A1} remains non-optimal, while it is optimal in the nonadaptive one.

\begin{figure*}[ht!]
\begin{center}
\vspace{-5ex}
\includegraphics[width = 1.01\textwidth]{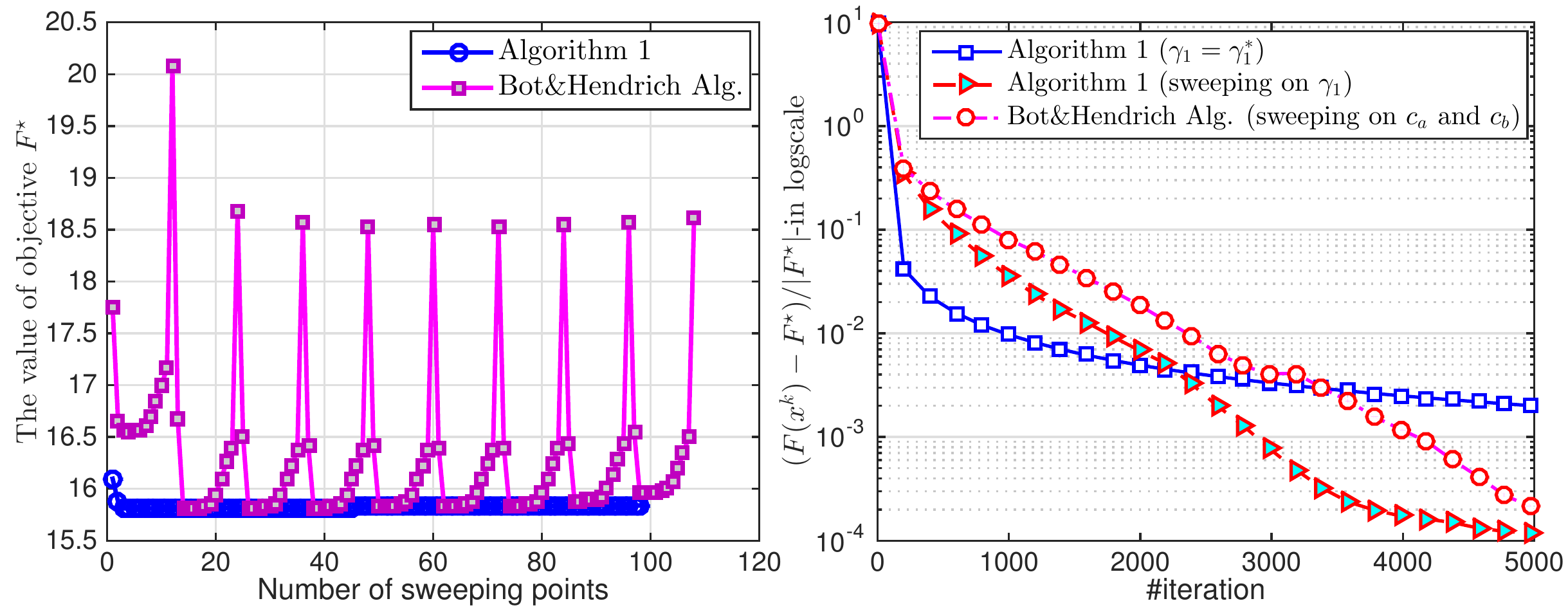}
\vspace{-3ex}
\caption{\footnotesize Comparison of Algorithm~\ref{alg:A1} and  \texttt{Bo\c{t}\&Hendrich~Alg.} (Left: the objective values vs. the number of sweeping points, ~~Right: Convergence of the relative objective residual).}\label{fig:compare_sqrtLasso}
\end{center}
\vspace{-6ex}
\end{figure*}

Finally, we compare Algorithm~\ref{alg:A1} with the variable smoothing algorithm in \cite{boct2012variable} (\texttt{Bo\c{t}\&Hendrich~Alg.}). 
Whlile the first term $f(\xb) := \Vert\Ab\xb - \bb\Vert_2$ is smoothened as in Algorithm \ref{alg:A1}, we smooth the second term $g(\xb) := \lambda\Vert\xb\Vert_1$ as 
\vspace{-0.75ex}
\begin{equation*}
g_{\beta}(\xb) := \max_{\vb}\set{\iprods{\xb,\vb} - (\beta/2)\Vert \vb\Vert_2^2 : \Vert\vb\Vert_{\infty} \leq 1}.
\vspace{-0.75ex}
\end{equation*} 
Then, we  update $\gamma_k$ and $\beta_k$ as $\gamma_k = \frac{1}{c_a(k+1)}$ and $ \beta_k = \frac{1}{c_b(k+1)}$, respectively as suggested in \cite{boct2012variable}, where $c_a$ and $c_b$ are two appropriate constants.

We compare Algorithm~\ref{alg:A1} and \texttt{Bo\c{t}\&Hendrich~Alg.} on a problem instance of  size $(p, n, s) = (1000, 350, 100)$, where the data is generated as in the previous tests. 
To find an appropriate value of $c_a$ and $c_b$, we sweep $c_a\in [10, 5000]$ simultaneously with $c_b \in [0.001, 500]$.
We obtain $c_a = 51$ and $c_b = 49$. 
For Algorithm~\ref{alg:A1}, we consider two cases. In the first case, we set $\gamma_1 = \gamma_1^{*} = 129.5505$ computed from the worst-case bound, while in the second case, we also sweep $\gamma_1\in [10, 1000]$ to find an appropriate value $\gamma_1 = 51$.
The results of both algorithms are plotted in Figure \ref{fig:compare_sqrtLasso} for $5000$ iterations.

Figure~\ref{fig:compare_sqrtLasso} (left) shows that the objective value produced by Algorithm \ref{alg:A1} does not vary much when $\gamma_1\in [10, 1000]$, while, in \texttt{Bo\c{t}\&Hendrich~Alg.}, the objective value changes rapidly when we sweep on $c_a$ and $c_b$ simultaneously. Hence, it is unclear how to choose an appropriate value for $c_a$ and $c_b$ without sweeping.
Figure \ref{fig:compare_sqrtLasso} (right) shows the convergence behavior of both algorithms. Without sweeping, Algorithm~\ref{alg:A1} has a good empirical convergence rate in the early iterations. With sweeping, both algorithms perform better in the later iterations. 
Algorithm~\ref{alg:A1} has a better performance than \texttt{Bo\c{t}\&Hendrich~Alg.}.

\vspace{-3.5ex}
\subsection{Image deblurring with the $\ell_1$ or $\ell_2$-data fidelity function}
\vspace{-2ex}
We consider an image deblurring problem using the $\ell_{\alpha}$-norm fidelity term as 
\vspace{-0.5ex}
\begin{equation}\label{eq:im_debulrring}
\min_{\Xb} \set{ F(\Xb) := \norm{\Ac(\Xb) - \bb}_{\alpha} + \lambda\norm{W\Xb}_1 : \Xb \in \R^{m\times q} },
\vspace{-0.5ex}
\end{equation}
where $\alpha \in \set{1, 2}$, $\Ab : \R^p\to\R^p$ ($p =m\times q$) is a blurring kernel, $\bb$ is an observed noisy image, $W : \R^p\to\R^p$ is the orthogonal Haar wavelet transform with four levels, $\lambda > 0$ is the regularizer parameter.

\begin{table}[ht!]
\begin{center}
\vspace{-1ex}
\caption{The PSNR values reported by the $8$ algorithmic variants on the $5$ test images}\label{tbl:PSNR_images}
\newcommand{\cell}[1]{#1}
\newcommand{\mblue}[1]{{\color{blue}#1}}
\newcommand{\celb}[1]{{\color{blue}#1}}
\newcommand{\celr}[1]{{\color{red}#1}}
\newcommand{\celn}[1]{{\color{Cerulean}#1}}
\rowcolors{2}{white}{black!05!white}
{\small 
\begin{tabular}{l|c|c|c|c|c}\toprule
{~~~~~~~~~~~~}Images{~~~}& \texttt{cameraman} & \texttt{barbara} & \texttt{lena} & \texttt{boat} & \texttt{house}  \\ \midrule
\multicolumn{6}{c}{PSNR of $4$ algorithms after $300$ iterations}\\ \midrule
Alg. \ref{alg:A1} ($\ell_1$, $\gamma_1=62$)       & \celn{26.2140} &\celn{26.8253} &\celn{27.1793} &\celn{26.4951} &\cell{30.9848} \\
Alg. \ref{alg:A1} ($\ell_1$, $\gamma_1$-sweeping) & \celb{26.2693} &\celb{27.0682} &\celb{27.5440} &\celb{26.5519} &\celb{31.6877} \\
Alg. \ref{alg:A1} ($\ell_2$, $\gamma_1=62$)       & \cell{26.2128} &\cell{26.8232} &\cell{27.1782} &\cell{26.4923} &\cell{30.4126} \\
Alg. \ref{alg:A1} ($\ell_2$, $\gamma_1$-sweeping) & \cell{26.2128} &\cell{26.8232} &\cell{27.1782} &\cell{26.4923} &\cell{30.4126} \\
Nes.  Alg. ($\ell_1$, $\gamma$-sweeping)               & \cell{25.0601} &\cell{26.1376} &\cell{26.3776} &\cell{25.2301} &\cell{30.2982} \\
Nes.  Alg. ($\ell_2$, $\gamma$-sweeping)               & \cell{25.0908} &\cell{26.1361} &\cell{26.3901} &\cell{25.2364} &\cell{30.4081} \\
BH  Alg. ($\ell_1$, $c_a$-sweeping)                                   & \cell{25.5784} &\cell{26.3421} &\cell{26.5916} &\cell{25.6025} &\celn{31.1606} \\
BH  Alg. ($\ell_2$, $c_a$-sweeping)                                   & \cell{25.4784} &\cell{26.4421} &\cell{26.5916} &\cell{25.6025} &\celn{31.1606} \\ \midrule
\multicolumn{6}{c}{PSNR of $4$ algorithms after $500$ iterations}\\ \midrule
Alg. \ref{alg:A1} ($\ell_1$, $\gamma_1=62$)       & \celn{27.0371} &\celn{27.6286} &\cell{28.1471} &\celn{27.3116} &\celn{32.1771} \\
Alg. \ref{alg:A1} ($\ell_1$, $\gamma_1$-sweeping) & \celb{27.1666} &\celb{27.8449} &\celb{28.2086} &\celb{27.4410} &\celb{32.8647} \\
Alg. \ref{alg:A1} ($\ell_2$, $\gamma_1=62$)       & \cell{27.0363} &\cell{27.6279} &\celn{28.1486} &\cell{27.3111} &\cell{32.1710} \\
Alg. \ref{alg:A1} ($\ell_2$, $\gamma_1$-sweeping) & \cell{27.0363} &\cell{27.6279} &\celn{28.1486} &\cell{27.3111} &\cell{32.1710} \\
Nes.  Alg. ($\ell_1$, $\gamma$-sweeping)               & \cell{25.0857} &\cell{26.1686} &\cell{26.4590} &\cell{26.1321} &\cell{30.4720} \\
Nes.  Alg. ($\ell_2$, $\gamma$-sweeping)               & \cell{25.0845} &\cell{26.1686} &\cell{26.4582} &\cell{25.2265} &\cell{30.4718} \\
BH  Alg. ($\ell_1$, $c_a$-sweeping)                         & \cell{26.5030} &\cell{27.1588} &\cell{27.1630} &\cell{27.0277} &\cell{31.8824} \\
BH  Alg. ($\ell_2$, $c_a$-sweeping)                         & \cell{26.5030} &\cell{27.1588} &\cell{27.1630} &\cell{27.0277} &\cell{31.8824} \\ \midrule
\multicolumn{6}{c}{PSNR of $4$ algorithms after $1000$ iterations}\\ \midrule
Alg. \ref{alg:A1} ($\ell_1$, $\gamma_1=62$)       & \celn{27.4774} &\cell{27.8353} &\cell{28.4224} &\celn{27.6596} &\cell{32.9985} \\
Alg. \ref{alg:A1} ($\ell_1$, $\gamma_1$-sweeping)  & \celb{27.3291} &\celb{27.8659} &\celb{28.4040} &\celb{27.9482} &\celb{33.2038} \\
Alg. \ref{alg:A1} ($\ell_2$, $\gamma_1=62$)        & \cell{27.2524} &\cell{27.8070} &\celn{28.4774} &\cell{27.5268} &\celn{33.1879} \\
Alg. \ref{alg:A1} ($\ell_2$, $\gamma_1$-sweeping)  & \cell{27.2524} &\cell{27.8070} &\celn{28.4774} &\cell{27.5268} &\celn{33.1879} \\
Nes.  Alg. ($\ell_1$, $\gamma$-sweeping)                & \cell{25.0870} &\cell{26.1691} &\cell{26.4602} &\cell{26.1371} &\cell{30.4698} \\
Nes.  Alg. ($\ell_2$, $\gamma$-sweeping)                & \cell{25.0867} &\cell{26.1690} &\cell{26.4600} &\cell{25.2267} &\cell{30.4700} \\
BH  Alg. ($\ell_1$, $c_a$-sweeping)                         & \cell{27.1128} &\celn{27.8391} &\cell{27.9327} &\cell{27.3487} &\cell{32.6715} \\
BH  Alg. ($\ell_2$, $c_a$-sweeping)                         & \cell{27.1723} &\cell{27.8205} &\cell{27.9327} &\cell{27.3143} &\cell{32.6715} \\ 
\bottomrule
\end{tabular}
}
\end{center}
\vspace{-7ex}
\end{table}

We now apply Algorithm~\ref{alg:A1} (\texttt{Alg.~\ref{alg:A1}}) to solve problem \eqref{eq:im_debulrring} and compare it with the nonadaptive variant (\texttt{Nes.~Alg.}) and Bo\c{t} \& Hendrich's algorithm (\texttt{BH~Alg.}) in \cite{boct2012variable}. 
Since $\Ab$ is orthogonal, we can use the quadratic smoothing function as $b_{\Uc}(\Xb) := (1/2)\norm{\Xb}_F^2$.
With this choice, we can compute the gradient of $\ub^{\ast}_{\gamma}(\Xb)$ defined by \eqref{eq:uast_beta} as $\ub^{\ast}_{\gamma}(\Xb) = \mathrm{proj}_{\mathcal{B}^{\ast}_{\alpha}}(\gamma^{-1}(\Ac(\Xb) - \bb))$, where $\mathrm{proj}_{\mathcal{B}_{\alpha}}$ is the projection onto the  dual norm ball $\mathcal{B}^{\ast}_{\alpha}$ of the $\ell_{\alpha}$-norm.

We test three algorithms on the five images: \texttt{cameraman}, \texttt{barbara}, \texttt{lena}, \texttt{boat} and \texttt{house} widely used in the literature.
The noisy images are generated as in \cite{Beck2009}.
Although we use the non-smooth $\ell_{\alpha}$-norm function with $\alpha=1$ or $\alpha = 2$, the regularization parameter $\lambda$ is set to $\lambda := 10^{-4}$ as suggested in \cite{Beck2009}, but it still provides the best recovery compared to other values in all $5$ images. 

While we fix $\gamma_1 = 62$ in Algorithm~\ref{alg:A1} which is roughly computed from the worst-case bound, we sweep $\gamma$ and $c_a$ (see Subsection \ref{subsec:sqrtLasso}) in $[0.0001, 1000]$ to choose the best possible value for \texttt{Nes. ~Alg.} and \texttt{BH~Alg.} in each image (with $300$ iterations). 
We also set $c_b = c_a$ as suggested in \cite{boct2012variable}.
For \texttt{Nes.~Alg.}, we have $\gamma = 1$ in the \texttt{boat} image, while in the other $4$ images, $\gamma = 2.5$ is the best value.
For \texttt{BH~Alg.}, we have $c_a = 0.005$ in the \texttt{cameraman}, \texttt{barbara} and  \texttt{boat} images, and $c_a = 0.0025$ in the  \texttt{lena} and \texttt{house} images.
The PSNR (Peak Signal to Noise Ratio \cite{Beck2009}) of the $8$ algorithms are reported in Table~\ref{tbl:PSNR_images}.

It shows that the nonsmooth $\ell_1$-norm objective produces slightly better recovery images in terms of PSNR than the $\ell_2$-norm objective in many cases for Algorithm~\ref{alg:A1}, but it is not the case in  \texttt{Nes.~ Alg.} and  \texttt{BH~Alg.} 
In addition, Algorithm \ref{alg:A1} is superior to \texttt{Nes. ~Alg.} in all cases, and is also better than \texttt{BH~Alg.} in the majority of the test.
We note that the complexity-per-iteration of the four algorithms are essentially the same, while our new adaptive strategy produces better solutions in terms of PSNR than the other two methods.
In addition, our algorithm significantly improves the PSNR if we run it further, while the nonadaptive variant does not make any clear progress on the PSNR value if we continue running it.
If we sweep the values of $\gamma_1$ in Algorithm~\ref{alg:A1} ($\gamma_1$-sweeping), we can also improve the results of this algorithm.

\vspace{-1.5ex}
\begin{acknowledgements}
\vspace{-2ex}
This research was supported  by NSF, Grant No. IPF 16-4829.
\end{acknowledgements}

\appendix
\vspace{-6ex}
\section{Appendix: The proof of technical results}
\vspace{-2.5ex}
This appendix provides the full proof of the technical results  presented in  the main text.

\vspace{-5ex}
\subsection{The proof of Lemma~\ref{le:descent_inequality}: Descent property of the proximal gradient step}\label{apdx:le:descent_inequality}
\vspace{-2ex}
By using \eqref{eq:pro_cond1} with $f_{\gamma}(\xb) := \varphi^{\ast}_{\gamma}(\Ab^{\top}\xb)$, $\nabla{f_{\gamma}}(\bar{\xb}) = \Ab\nabla{\varphi_{\gamma}^{\ast}}(\Ab^{\top}\bar{\xb})$, $\zb := \Ab^{\top}\xb$, $\bar{\zb} := \Ab^{\top}\bar{\xb}$, and $\norm{\Ab^{\top}(\xb - \bar{\xb})} \leq \norm{\Ab}\norm{\xb - \bar{\xb}}$ we can show that
\begin{equation*}
\frac{\gamma}{2}\norm{\ub^{\ast}_{\gamma}(\Ab^{\top}\xb) - \ub^{\ast}_{\gamma}(\Ab^{\top}\bar{\xb})}^2 \leq f_{\gamma}(\xb) - f_{\gamma}(\bar{\xb}) - \iprods{\nabla{f_{\gamma}}(\xb), \xb - \bar{\xb}} \leq \frac{\norm{\Ab}^2}{2\gamma}\norm{\xb - \hat{\xb}}^2.
\end{equation*}
Using this estimate, we can show that the proof of \eqref{eq:key_est21} can be done similarly as in \cite{Tran-Dinh2014a}.
\vspace{-2ex}
\Eproof

\vspace{-4ex}
\subsection{The proof of Lemma \ref{le:key_estimate1}: Key estimate}\label{apdx:le:key_estimate1}
\vspace{-2ex}
We first substitute $\beta = \frac{\gamma_{k\!+\!1}}{\norm{\Ab}^2}$ into  \eqref{eq:key_est21} and using \eqref{eq:key_est21a} to obtain
\begin{equation*}
\begin{array}{ll}
F_{\gamma_{k\!+\!1}}(\xb^{k\!+\!1}\!) \!&\leq\! F_{\gamma_{k\!+\!1}}(\xb^k) - \frac{\gamma_{k\!+\!1}}{2}\norm{ \ub^{\ast}_{\gamma_{k\!+\!1}}(\Ab^{\top}\xb^k) - \ub^{\ast}_{\gamma_{k\!+\!1}}(\Ab^{\top}\xhat^k) }^2 \nonumber\vspace{1ex}\\
&+ \frac{\norm{\Ab}^2}{\gamma_{k\!+\!1}}\iprod{ \xb^{k\!+\!1} - \xhat^k, \xb - \xbar^k}  -\frac{\norm{\Ab}^2}{2\gamma_{k\!+\!1}}\norm{\xhat^k - \xb^{k\!+\!1}}^2.
\end{array}
\end{equation*}
Multiplying this inequality by $(1-\tau_k)$ and   \eqref{eq:key_est21} by $\tau_k$, and  summing up the results we obtain
\vspace{-0.5ex}
\begin{align*}
\begin{array}{ll}
F_{\gamma_{k\!+\!1}}(\xb^{k\!+\!1}\!) \!&\leq\! (1 \!-\! \tau_k)F_{\gamma_{k\!+\!1}}(\xb^k) \!+\! \tau_k\hat{\ell}^k_{\gamma_{k\!+\!1}}(\xb) 
-\frac{(1\!-\!\tau_k)\gamma_{k\!+\!1}}{2}\norm{ \ub^{\ast}_{\gamma_{k\!+\!1}}(\Ab^{\top}\xb^k) - \ub^{\ast}_{\gamma_{k\!+\!1}}(\Ab^{\top}\xhat^k) }^2\nonumber\vspace{1ex}\\
&+ \frac{\norm{\Ab}^2}{\gamma_{k\!+\!1}}\iprod{\xhat^k - \xb^{k\!+\!1}, \xhat^k - (1-\tau_k)\xb^k - \tau_k\xb}  -\frac{\norm{\Ab}^2}{2\gamma_{k\!+\!1}}\norm{\xhat^k - \xb^{k\!+\!1}}^2,
\end{array}
\end{align*}
where $\hat{\ell}_{\gamma}^k(\xb) := f_{\gamma}(\xhat^k) + \iprods{\nabla{f_{\gamma}}(\xhat^k), \xb - \xhat^k} + g(\xb)$.

From \eqref{eq:acc_grad_scheme}, we have $\tau_k\tilde{\xb}^k = \xhat^k - (1-\tau_k)\xb^k$, we can write this inequality as
\begin{align}\label{eq:proof_est3}
F_{\gamma_{k\!+\!1}}(\xb^{k\!+\!1}) &\!\leq\! (1 \!-\! \tau_k)F_{\gamma_{k\!+\!1}}(\xb^k) \!+\! \tau_k\hat{\ell}^k_{\gamma_{k\!+\!1}}\!(\xb) 
-\frac{(1\!-\!\tau_k)\gamma_{k\!+\!1}}{2}\norm{ \ub^{\ast}_{\gamma_{k\!+\!1}}(\Ab^{\top}\xb^k) - \ub^{\ast}_{\gamma_{k\!+\!1}}(\Ab^{\top}\xhat^k) }^2\nonumber\\
& + \frac{\norm{\Ab}^2\tau_k^2}{2\gamma_{k\!+\!1}}\Big[ \norm{\tilde{\xb}^k - \xb}^2 - \Vert \tilde{\xb}^k - \tau_k^{-1}(\xhat^k - \xb^{k\!+\!1}) - \xb \Vert^2 \Big].
\end{align}
Using \eqref{eq:key_est2} with $\gamma := \gamma_{k\!+\!1}$, $\hat{\gamma} := \gamma_k$ and $\zb := \Ab^{\top}\xb^k$, we get
\vspace{-0.75ex}
\begin{align*}
f_{\gamma_{k\!+\!1}}(\xb^k) \leq f_{\gamma_k}(\xb^k) + (\gamma_k - \gamma_{k\!+\!1})b_{\Uc}(\uast_{\gamma_{k\!+\!1}}(\Ab^{\top}\xb^k)),
\vspace{-0.75ex}
\end{align*}
which leads to (\emph{cf}: $F_{\gamma} = f_{\gamma} + g$):
\vspace{-0.5ex}
\begin{align}\label{eq:proof_est4}
F_{\gamma_{k\!+\!1}}(\xb^k) \leq F_{\gamma_k}(\xb^k) + (\gamma_k - \gamma_{k\!+\!1})b_{\Uc}(\uast_{\gamma_{k\!+\!1}}(\Ab^{\top}\xb^k)).
\vspace{-0.5ex}
\end{align}
Next, we estimate $\hat{\ell}_{\gamma_{k\!+\!1}}^k$. Using the definition of $f_{\gamma}$ and $\nabla{f}_{\gamma}$, we can deduce
\begin{equation}\label{eq:proof_est5a}
\begin{array}{ll}
{\!\!\!\!}\hat{\ell}_{\gamma_{k\!+\!1}}^k(\xb) &:= f_{\gamma_{k\!+\!1}}(\xhat^k) + \iprod{\nabla{f_{\gamma_{k\!+\!1}}}(\xhat^k), \xb - \xhat^k} + g(\xb) \vspace{1ex}\\
&= \iprodb{\xhat^k, \Ab^{\top}\ub^{*}_{\gamma_{k\!+\!1}}(\Ab^{\top}\xhat^k)} - \varphi(\ub^{*}_{\gamma_{k\!+\!1}}(\Ab^{\top}\xhat^k)) - \gamma_{k\!+\!1}b_{\Uc}(\ub^{*}_{\gamma_{k\!+\!1}}(\Ab^{\top}\xhat^k)) \vspace{1ex}\\
&+ \iprodb{\xb - \xhat^k, \Ab\ub^{*}_{\gamma_{k\!+\!1}}(\Ab^{\top}\xhat^k)} + g(\xb) \vspace{1ex}\\
&= \iprodb{\xb, \Ab\ub^{*}_{\gamma_{k\!+\!1}}(\Ab^{\top}\xhat^k)} - \varphi(\ub^{*}_{\gamma_{k\!+\!1}}(\Ab^{\top}\xhat^k)) - \gamma_{k\!+\!1}b_{\Uc}(\ub^{*}_{\gamma_{k\!+\!1}}(\Ab^{\top}\xhat^k)) + g(\xb) \vspace{1ex}\\
&\leq \displaystyle\max_{\ub}\set{\iprodb{\xb, \Ab\ub}  - \varphi(\ub) : \ub \in\Uc } - \gamma_{k\!+\!1}b_{\Uc}(\ub^{*}_{\gamma_{k\!+\!1}}(\Ab^{\top}\xhat^k)) + g(\xb) \vspace{1ex}\\
&= F(\xb) - \gamma_{k\!+\!1}b_{\Uc}(\ub^{*}_{\gamma_{k\!+\!1}}(\Ab^{\top}\xhat^k)).
\end{array}
\end{equation}
Substituting $\tilde{\xb}^{k\!+\!1} := \tilde{\xb}^k - \frac{1}{\tau_k}(\xhat^k - \xb^{k\!+\!1})$ from the third line of \eqref{eq:acc_grad_scheme} together with \eqref{eq:proof_est4}, and \eqref{eq:proof_est5a} into \eqref{eq:proof_est3}, we can derive
\vspace{-0.5ex}
\begin{align*} 
\begin{array}{ll}
F_{\gamma_{k\!+\!1}}(\xb^{k\!+\!1}) \leq (1\!-\! \tau_k)F_{\gamma_{k}}(\xb^k) \!+\! \tau_kF(\xb) \!+\! \frac{\norm{\Ab}^2\tau_k^2}{2\gamma_{k\!+\!1}}\left[\norm{\tilde{\xb}^k \!\!-\! \xb}^2 - \norm{\tilde{\xb}^{k\!+\!1} \!\!-\! \xb}^2\right] - R_k,
\end{array}
\end{align*}
which is indeed \eqref{eq:key_estimate31}, where $R_k$ is given by  \eqref{eq:Rk_term}.

Finally, we prove \eqref{eq:Rk_lower_bound}.
Indeed, using  the strong convexity and the $L_b$-smoothness of $b_{\Uc}$, we can lower bound
\begin{align*}
\begin{array}{ll}
R_k &\geq \frac{\tau_k\gamma_{k\!+\!1}}{2}\norm{\uast_{\gamma_{k\!+\!1}}(\Ab^{\top}\xhat^k) \!-\! \ubar^c}^2 + \frac{(1-\tau_k)\gamma_{k\!+\!1}}{2}\norm{\uast_{\gamma_{k\!+\!1}}(\Ab^{\top}\xb^k) - \uast_{\gamma_{k\!+\!1}}(\Ab^{\top}\xhat^k)}^2\nonumber\vspace{1ex}\\
&-\frac{L_b}{2}(1-\tau_k)(\gamma_k - \gamma_{k\!+\!1})\norm{\uast_{\gamma_{k\!+\!1}}(\Ab^{\top}\xb^k) - \ubar^c}^2.
\end{array}
\end{align*}
Letting $\hat{\vb}_k := \uast_{\gamma_{k\!+\!1}}(\Ab^{\top}\xhat^k) - \ubar^c$ and $ \vb_k := \uast_{\gamma_{k\!+\!1}}(\Ab^{\top}\xb^k) - \ubar^c$, we write $R_k$ as
\vspace{-0.5ex}
\begin{align*}
\begin{array}{ll}
2\gamma_{k\!+\!1}^{-1}R_k &\geq \tau_k\norm{\hat{\vb}_k }^2 + (1-\tau_k)\norm{\hat{\vb}_k - \vb_k}^2 - (1-\tau_k)(\gamma_{k\!+\!1}^{-1}\gamma_k - 1)L_b\norm{\vb_k}^2 \vspace{1ex} \\
&=\norm{\hat{\vb}_k}^2 - 2(1-\tau_k)\iprods{\hat{\vb}_k ,\vb_k } + (1-\tau_k)\big[1 - (\gamma_{k\!+\!1}^{-1}\gamma_k - 1)L_b\big]\norm{\vb_k}^2 \vspace{1ex} \\
&= \norm{\hat{\vb}_k - (1-\tau_k)\vb_k}^2 + (1-\tau_k)\left[\tau_k - (\gamma_{k\!+\!1}^{-1}\gamma_k - 1)L_b\right]\norm{\vb_k}^2 \vspace{1ex} \\
&\geq (1-\tau_k)\left[\tau_k - (\gamma_{k\!+\!1}^{-1}\gamma_k - 1)L_b\right]\norm{\vb_k}^2,
\end{array}
\end{align*}
which obviously implies \eqref{eq:Rk_lower_bound}.
\Eproof

\vspace{-5ex}
\subsection{The proof of Lemma \ref{le:key_estimate2}: The choice of parameters}\label{apdx:le:key_estimate2}
\vspace{-3ex}
First, using the update rules of $\tau_k$ and $\gamma_k$ in \eqref{eq:param_update_rules}, we can express the quantity $m_k$ as
\vspace{-0.5ex}
\begin{align*}
m_k &:= \frac{\gamma_{k\!+\!1}(1-\tau_k)\left[\gamma_{k\!+\!1}\tau_k - L_b(\gamma_k - \gamma_{k\!+\!1})\right]}{\tau_k^2}  = -\frac{\gamma_1^2\bar{c}^2\left[(L_b-1)(k + \bar{c}) + 1\right]}{(k+\bar{c})^2}.
\vspace{-0.5ex}
\end{align*}
Moreover, it follows from the properties of $b_{\Uc}$ that
\begin{equation*}
\frac{1}{2}\Vert \ub^{*}_{\gamma_{k\!+\!1}}(\Ab^{\top}\xb^k) -\ubar^c\Vert^2 \leq b_{\Uc}(\ub^{*}_{\gamma_{k\!+\!1}}(\Ab^{\top}\xb^k)) \leq D_{\Uc}.
\end{equation*}
Multiplying the lower bound \eqref{eq:Rk_lower_bound} by $\frac{\gamma_{k\!+\!1}}{\tau_k^2}$ and combining the result  with the last inequality and the estimate of $m_k$, we obtain the first lower bound in \eqref{eq:Rk_lower_bound2}.

Next, using the update rules \eqref{eq:param_update_rules} of $\tau_k$ and $\gamma_k$, we have $\frac{(1-\tau_k)\gamma_{k\!+\!1}}{\tau_k^2} = \frac{(k+\bar{c}-1)(k+\bar{c})^2\gamma_1\bar{c}}{(k+\bar{c})(k+\bar{c})}  = \frac{\gamma_1\bar{c}(k+\bar{c}-1)^2}{(k+\bar{c}-1)} = \frac{\gamma_k}{\tau_{k-1}^2}$, which is the second equality in \eqref{eq:Rk_lower_bound2}. 

Using \eqref{eq:key_estimate31} with the lower bound of $R_k$ from \eqref{eq:Rk_lower_bound2}, we have
\begin{equation}\label{eq:variant2_proof1}
\frac{\gamma_{k\!+\!1}}{\tau_k^2}\Delta{F}_{k\!+\!1} \!+\! \frac{\norm{\Ab}^2}{2}\norm{\tilde{\xb}^{k\!+\!1} \!\!\!\!-\! \xopt}^2 \!\leq\! \frac{(1 \!-\! \tau_k)\gamma_{k\!+\!1}}{\tau_k^2}\Delta{F}_k \!+\!  \frac{\norm{\Ab}^2}{2}\norm{\tilde{\xb}^k \!\!-\! \xopt}^2 + s_kD_{\Uc},
\end{equation}
where $\Delta{F}_k := F_{\gamma_k}(\xb^k) - \Fopt$ and $s_k := \frac{\gamma_1^2\bar{c}^2\left[(L_b\!-\!1)(k \!+\! \bar{c}) \!+\! 1\right]}{(k \!+\! \bar{c})^2}$.
Using this inequality and  the relation $\frac{(1-\tau_k)\gamma_{k\!+\!1}}{\tau_k^2} = \frac{\gamma_k}{\tau_{k-1}^2}$ in \eqref{eq:Rk_lower_bound2}, we can easily show that
\begin{align*} 
\frac{\gamma_{k\!+\!1}}{\tau_k^2}\Delta{F}_{k\!+\!1} \!+\! \frac{\norm{\Ab}^2}{2}\norm{\tilde{\xb}^{k\!+\!1} \!\!\!\!-\! \xopt}^2 \!\leq\! \frac{\gamma_{k}}{\tau_{k\!-\!1}^2}\Delta{F}_k \!+\!  \frac{\norm{\Ab}^2}{2}\norm{\tilde{\xb}^k \!\!-\! \xopt}^2 + s_kD_{\Uc}.
\end{align*}
By induction, we obtain from the last inequality that 
\begin{equation}\label{eq:variant2_proof5}
\frac{\gamma_{k\!+\!1}}{\tau_k^2}\Delta{F}_{k\!+\!1} \!+\! \frac{\norm{\Ab}^2}{2}\norm{\tilde{\xb}^{k\!+\!1} \!\!\!\!-\! \xopt}^2 \!\leq\! \frac{(1\!-\!\tau_0)\gamma_1}{\tau_0^2}\Delta{F}_0 \!+\!  \frac{\norm{\Ab}^2}{2}\norm{\tilde{\xb}^0 \!\!-\! \xopt}^2 \!+\!  S_kD_{\Uc},
\end{equation}
which implies \eqref{eq:key_estimate2_new}, where $S_k := \sum_{i=0}^ks_k = \gamma_1^2\bar{c}^2\sum_{i=0}^k\frac{\left[(L_b\!-\!1)(i +  \bar{c}) + 1\right]}{(i  + \bar{c})^2}$.

Finally, to prove \eqref{eq:Sk_bound}, we use two elementary inequalities $\sum_{i=1}^{k+\bar{c}}\frac{1}{i} < 1 + \ln(k+\bar{c})$ and $\sum_{i=0}^k\frac{1}{(i+\bar{c})^2} \leq \frac{1}{\bar{c}^2} + \sum_{i=1}^k\frac{1}{(i+\bar{c}-1)(i+\bar{c})} < \frac{1}{\bar{c}^2} + \frac{1}{\bar{c}}$.
\Eproof

\vspace{-3.5ex}
\subsection{The proof of Corollary~\ref{co:update_param2}: The smooth accelerated gradient method}\label{apdx:co:update_param2}
\vspace{-2.5ex}
First, it is similar to the proof of \eqref{eq:variant2_proof1}, we can derive
\begin{equation*} 
\frac{\gamma_{k\!+\!1}}{(L_g\gamma_{k+1} + \norm{\Ab}^2)\tau_k^2}\Delta{F}_{k\!+\!1} \!+\! \frac{1}{2}\norm{\tilde{\xb}^{k\!+\!1} \!\!\!\!-\! \xopt}^2 \!\leq\! \frac{(1 \!-\! \tau_k)\gamma_{k\!+\!1}}{(L_g\gamma_{k+1} + \norm{\Ab}^2)\tau_k^2}\Delta{F}_k \!+\!  \frac{1}{2}\norm{\tilde{\xb}^k \!\!-\! \xopt}^2 + \hat{s}_kD_{\Uc},
\end{equation*}
where $\Delta{F}_k := F_{\gamma_{k}}(\xb^k) - F^{\star}$, and $\hat{s}_k := \frac{\gamma_{k\!+\!1}(1-\tau_k)\left[L_b(\gamma_k - \gamma_{k\!+\!1}) - \gamma_{k\!+\!1}\tau_k\right]}{\tau_k^2(L_g\gamma_{k+1}+\norm{\Ab}^2)}$.

Next, we impose  condition $\frac{(1-\tau_k)\gamma_{k+1}}{\tau_k^2(L_g\gamma_{k+1} + \norm{\Ab}^2)} = \frac{\gamma_k}{\tau_{k-1}^2(L_g\gamma_k + \norm{\Ab}^2)}$ and choose $\tau_k = \frac{1}{k+1}$.
Then, we can show from the last condition that $\gamma_{k+1} = \frac{k\gamma_k\norm{\Ab}^2}{L_g\gamma_k + \norm{\Ab}^2(k+1)}$.
Now, we show that $\gamma_k \leq \frac{\gamma_1}{k+1}$. 
Indeed, we have $\frac{1}{\gamma_{k+1}} =  \left(\frac{k+1}{k}\right)\frac{1}{\gamma_k} + \frac{L_g}{\norm{\Ab}^2k}  \geq  \left(\frac{k+1}{k}\right)\frac{1}{\gamma_k}$, 
which implies that $\gamma_{k+1} \leq \frac{k}{k+1}\gamma_k$. 
By induction, we get $\gamma_{k+1} \leq \frac{\gamma_1}{k+1}$.
On the other hand, assume that $\frac{1}{\gamma_{k+1}} =  \left(\frac{k+1}{k}\right)\frac{1}{\gamma_k} + \frac{L_g}{\norm{\Ab}^2k}  \leq \frac{1}{\gamma_k} \left(\frac{k}{k-1}\right)$ for $k\geq 2$.
This condition leads to $\gamma_k \leq \frac{\norm{\Ab}^2}{L_g(k-1)}$. 
Using $\gamma_k \leq \frac{\gamma_1}{k} = \frac{\norm{\Ab}^2}{L_gk}$ due to the choice of $\gamma_1$, we can show that $\gamma_k \leq \frac{\norm{\Ab}^2}{L_g(k-1)}$.
Hence, with the choice $\gamma_1 := \frac{\norm{\Ab}^2}{L_g}$, the estimate $\frac{1}{\gamma_{k+1}} \leq  \frac{1}{\gamma_k} \left(\frac{k}{k-1}\right)$ and the update rule of $\gamma_k$ eventually imply
\begin{equation*}
\frac{\gamma_1\norm{\Ab}^2}{(L_g\gamma_1 + 2\norm{\Ab}^2)k} = \frac{\gamma_2}{k} \leq \gamma_{k+1} \leq \frac{\gamma_1}{k+1},~~\forall k\geq 1.
\end{equation*}
This condition leads to  $\frac{\tau_{k-1}^2(L_g\gamma_k + \norm{\Ab}^2)}{\gamma_k} = L_g\tau_{k-1}^2 + \frac{\tau_{k-1}^2}{\gamma_k}\norm{\Ab}^2 \leq \frac{L_g}{k^2} + \frac{3L_g(k-1)}{k^2} = \frac{3L_g}{k}$.
Using the estimates of $\tau_k$ and $\gamma_k$, we can easily show that $\hat{s}_k \leq \frac{L_b\norm{\Ab}^2}{L_g^2k(k+2)} + \frac{L_b\norm{\Ab}^2}{L_g^2(k+1)(k+2)} + \frac{(L_b-1)\norm{\Ab}^2k}{L_g(k+1)(k+2)}$.
Hence, we can show that
\begin{equation*}
\hat{S}_k := \sum_{i=0}^k\hat{s}_i \leq \frac{2L_b\norm{\Ab}^2}{L_g^2} +  \frac{(L_b-1)\norm{\Ab}^2}{L_g^2}\sum_{i=0}^k\frac{1}{i+1} = \frac{2L_b\norm{\Ab}^2}{L_g^2} +  \frac{(L_b-1)\norm{\Ab}^2}{L_g^2}\left(\ln(k) + 1\right).
\end{equation*}
Using this estimate, we can show that 
\begin{equation*}
F_{\gamma_k}(\xb^k) - F^{\star} \leq \frac{3L_g}{2k}\norm{\xb^0 - \xopt}^2 + \frac{2L_b\norm{\Ab}^2}{L_g^2k}D_{\Uc} +  \frac{(L_b-1)\norm{\Ab}^2}{L_g^2k}\left(\ln(k) + 1\right)D_{\Uc}.
\end{equation*}
Finally, using the bound \eqref{eq:key_est1} and $\gamma_k \leq \frac{\norm{\Ab}^2}{L_g(k+1)} < \frac{\norm{\Ab}^2}{L_gk}$, we obtain \eqref{eq:convergence5}.
\Eproof

\vspace{-3.5ex}
\subsection{The proof of Theorem \ref{th:primal_recovery}: Primal solution recovery}\label{apdx:th:primal_recovery}
\vspace{-2.5ex}
Let $\Delta{F}_k := F_{\gamma_{k}}(\xb^k) - \Fopt$. 
Then, by \eqref{eq:key_est1}, we have $\Delta{F}_k \geq F(\xb^k) - \Fopt - \gamma_kD_{\Uc} \geq -\gamma_kD_{\Uc}$.
Similar to the proof of Lemma~\ref{le:key_estimate2}, we can prove that 
\begin{align}\label{eq:proof_lm41_2} 
\frac{\gamma_{i\!+\!1}}{\tau_i^2}\Delta{F}_{i\!+\!1} \leq \frac{\gamma_i}{\tau_{i-1}^2}\Delta{F}_i + \frac{\gamma_{i\!+\!1}}{\tau_i}\Delta{\hat{\ell}}_{\gamma_{i\!+\!1}}(\xb) + \frac{\norm{\Ab}^2}{2}\left(\norm{\tilde{\xb}^i - \xb}^2 - \norm{\tilde{\xb}^{i\!+\!1} - \xb}^2\right) + s_iD_{\Uc},
\end{align}
where $s_i := \frac{\left[(L_b\!-\!1)(i +  \bar{c}) + 1\right]}{(i  + \bar{c})^2}$ as in the proof of Lemma \ref{le:key_estimate2}, and $\Delta{\hat{\ell}}_{\gamma_{k\!+\!1}}^k(\xb) =    \iprodb{\xb, \Ab\ub^{*}_{\gamma_{k\!+\!1}}(\xhat^k)} - \varphi(\ub^{*}_{\gamma_{k\!+\!1}}(\xhat^k)) + g(\xb) - \Fopt =  \iprodb{\xb, \Ab\ub^{*}_{\gamma_{k\!+\!1}}(\xhat^k) - \bb} - \varphi(\ub^{*}_{\gamma_{k\!+\!1}}(\xhat^k))  + s_{\Kc}(\xb) - \Fopt$.
Summing up this inequality from $i=1$ to $i=k$ and using $\tau_0 = 1$ and $\tilde{\xb}^0 =\xb^0$, we obtain
\begin{align}\label{eq:proof_lm41_3a}
{\!\!\!}\frac{\gamma_{k\!+\!1}}{\tau_k^2}\Delta{F}_{k\!+\!1} &\leq \sum_{i=1}^k\frac{\gamma_{i+1}}{\tau_i}\Delta{\hat{\ell}}_{\gamma_{i\!+\!1}}(\xb) \!+\! \frac{\norm{\Ab}^2}{2}\left(\norm{\tilde{\xb}^1 \!-\! \xb}^2 \!-\! \norm{\tilde{\xb}^{k\!+\!1} \!\!- \xb}^2\right) +\gamma_1\Delta{F}_1 + S_kD_{\Uc},{\!\!\!}
\end{align}
where $S_k := \sum_{i=1}^ks_i$.
Now, using again \eqref{eq:proof_lm41_2} with $k=1$, $\xb^0 = \tilde{\xb}^0$ and $\tau_0 = 1$, we get $\gamma_1\Delta{F}_1 \leq \gamma_1\tau_0\Delta{\hat{\ell}}_{\gamma_1}(x) + \frac{\norm{\Ab}^2}{2}\left(\Vert\xb^0 - \xopt\Vert^2 - \Vert\tilde{\xb}^1 - \xopt\Vert^2\right)$.
Using this into \eqref{eq:proof_lm41_3a}, one yields
\begin{align}\label{eq:proof_lm41_3} 
\frac{\gamma_{k\!+\!1}}{\tau_k^2}\Delta{F}_{k\!+\!1} &\leq\sum_{i=0}^k\frac{\gamma_{i+1}}{\tau_i}\left( \iprodb{\xb, \Ab\ub^{*}_{\gamma_{i\!+\!1}}(\xhat^i) - \bb} \!-\! \varphi(\ub^{*}_{\gamma_{i\!+\!1}}(\xhat^i)) + s_{\Kc}(\xb)- \Fopt\right)\nonumber\\
&   + \frac{\norm{\Ab}^2}{2}\norm{\xb^0 - \xb}^2 + S_kD_{\Uc}.
\end{align}
Combining $\ubar^k$ defined by \eqref{eq:averaging_scheme} with $w_i :=\frac{\gamma_{i+1}}{\tau_i}$, and the convexity of $\varphi$, we have
\begin{equation*}
\sum_{i=0}^k\frac{\gamma_{i+1}}{\tau_i}\left( \iprodb{\xb, \Ab\ub^{*}_{\gamma_{i\!+\!1}}(\xhat^i) - \bb} \!-\! \varphi(\ub^{*}_{\gamma_{i\!+\!1}}(\xhat^i))\right) \leq \Gamma_k\left(\iprodb{\xb, \Ab\bar{\ub}^k - \bb} - \varphi(\ubar^k)\right).
\end{equation*}
Substituting this into \eqref{eq:proof_lm41_3} and then using  $\Delta{F}_{k\!+\!1} \geq -\gamma_{k\!+\!1}D_{\Uc}$ we get
\begin{align*} 
-\frac{\gamma_{k\!+\!1}^2}{\tau_k^2}D_{\Uc} 
\leq \Gamma_k\left(\iprodb{\xb, \Ab\bar{\ub}^k - \bb} - \varphi(\ubar^k) \!+\! s_{\Kc}(\xb) \!-\! \Fopt\right) \!+\! \frac{\norm{\Ab}^2}{2}\norm{\xb^0 \! -\! \xb}^2 + S_kD_{\Uc},
\end{align*}
which implies
\begin{align*} 
\Fopt \leq \iprodb{\xb, \Ab\bar{\ub}^k - \bb} - \varphi(\ubar^k)  + s_{\Kc}(\xb) + \frac{\norm{\Ab}^2}{2\Gamma_k}\norm{\xb^0 - \xb}^2 + \frac{D_{\Uc}}{\Gamma_k}\left(\frac{\gamma_{k\!+\!1}^2}{\tau_k^2} + S_k\right).
\end{align*}
By arranging this inequality, we get
\begin{equation}\label{eq:proof_lm41_3b} 
\inf_{\rb\in\Kc}\iprodb{\xb, \bb - \Ab\bar{\ub}^k - \rb} +  \varphi(\ubar^k) \leq -\Fopt + \frac{\norm{\Ab}^2}{2\Gamma_k}\norm{\xb^0 - \xb}^2 +  \frac{D_{\Uc}}{\Gamma_k}\left(\frac{\gamma_{k\!+\!1}^2}{\tau_k^2} + S_k\right),
\end{equation}
where we use the relation $-s_{\Kc}(\xb) = -\sup_{\rb\in\Kc}\iprods{\xb, \rb} = \inf_{\rb\in\Kc}\iprods{\xb, -\rb}$.
On the other hand, by the saddle point theory for the primal and dual problems \eqref{eq:constr_cvx} and \eqref{eq:constr_cvx_dual}, for any optimal solution $\xopt$, we can show that
\begin{align*}
-F^{\star} = \varphi^{\star} \leq \varphi(\ub) - \iprods{\xopt, \Ab\ub - \bb + \rb}, ~\forall\ub\in\Uc,~\rb\in\Kc.
\end{align*}
Since this inequality holds for any $\rb\in\Kc$ and $\ub\in\Uc$, by using  $\ub = \bar{\ub}^k$, it leads to
\begin{align}\label{eq:proof_lm41_4}
\inf_{\rb\in\Kc}\iprodb{\xopt, \Ab\ubar^k - \bb + \rb} - \varphi(\ubar^k) \leq F^{\star}.
\end{align}
Combining \eqref{eq:proof_lm41_3b}  and \eqref{eq:proof_lm41_4} yields
\begin{equation}\label{eq:proof_lm41_5}
\min_{\rb\in\Kc}\Big\{\iprodb{\xopt \!-\! \xb, \rb + \Ab\bar{\ub}^k - \bb} -  \frac{\norm{\Ab}^2}{2\Gamma_k}\norm{\xb^0 \!\!-\! \xb}^2\Big\} \leq  
\frac{D_{\Uc}}{\Gamma_k}\left(\frac{\gamma_{k\!+\!1}^2}{\tau_k^2} + S_k\right), ~~\forall\xb\in\R^p.
\end{equation}
Taking $\xb := \xb^0 - \norm{\Ab}^{-2}\Gamma_k(\Ab\bar{\ub}^k - \bb + \rb)$ for any $\rb\in\Kc$, we obtain from \eqref{eq:proof_lm41_5} that
\begin{align*}
\min_{\rb\in\Kc}\set{\frac{\Gamma_k}{\norm{\Ab}^2}\norm{\Ab\bar{\ub}^k + \rb - \bb}^2 + 2\iprodb{\Ab\bar{\ub}^k - \bb + \rb, \xopt - \xb^0}} \leq \frac{2D_{\Uc}}{\Gamma_k}\left(\frac{ \gamma_{k\!+\!1}^2}{\tau_k^2} + S_k\right),
\end{align*}
which implies (by the Cauchy-Schwarz inequality)
\begin{align*} 
\min_{\rb\in\Kc}\set{\Gamma_k\norm{\Ab\bar{\ub}^k  - \bb + \rb}^2 - 2\norm{\Ab}^2\norm{\Ab\bar{\ub}^k  - \bb + \rb}\norm{\xopt - \xb^0}} \leq \frac{2\norm{\Ab}^2D_{\Uc}}{\Gamma_k}\left(\frac{\gamma_{k\!+\!1}^2}{\tau_k^2} + S_k\right).
\end{align*}
By elementary calculations and $\dist{\bb - \Ab\bar{\ub}^k, \Kc} = \min\set{\norm{\Ab\bar{\ub}^k\!-\!\bb + \rb} : \rb\in\Kc}$, we can show from the last inequality that
\begin{align}\label{eq:proof_lm41_6}
\dist{\bb - \Ab\bar{\ub}^k, \Kc}  \leq \frac{\norm{\Ab}^2}{\Gamma_k}\Big[\norm{\xb^0 \!-\! \xopt} \!+\! \sqrt{\norm{\xb^0 \!-\! \xopt}^2 \!+\! \frac{2}{\norm{\Ab}^2}\Big( S_k+\frac{\gamma_{k\!+\!1}^2}{\tau_k^2}\Big)D_{\Uc}}\Big].
\end{align}
To prove the first estimate of \eqref{eq:primal_recovery2}, we use \eqref{eq:proof_lm41_3b} with $\xb = \boldsymbol{0}^p$ and $F^{\star} = -\varphi^{\star}$ to get
\begin{align}\label{eq:proof_lm41_6b}
\varphi(\ubar^k) - \varphi^{\star} \leq \frac{1}{\Gamma_k}\Big[\frac{\norm{\Ab}^2}{2}\norm{\xb^0}^2 + \Big(\frac{\gamma_{k\!+\!1}^2}{\tau_k^2} + S_k\Big)D_{\Uc}\Big].
\end{align}
Since we apply Algorithm $\mathrm{\ref{alg:A1}(c)}$ to solve the dual problem \eqref{eq:constr_cvx_dual} using $b_{\Uc}$ such that $L_b = 1$, we have $S^k \leq 2\gamma_1^2$. 
Then, by using $\gamma_{k\!+\!1} = \frac{\bar{c}\gamma_1}{k+\bar{c}}$,  and  $\tau_k := \frac{1}{k+\bar{c}}$, we can show that 
$\frac{\gamma_{k\!+\!1}^2}{\tau_k^2} = \gamma_1\bar{c}$.
Moreover, we also have $\Gamma_k := \sum_{i=0}^k\frac{\gamma_{i+1}}{\tau_i} = \gamma_1\bar{c}(k+1)$.
Using these estimates, and  $S_k  \leq 2\gamma_1^2$ from \eqref{le:key_estimate2} into \eqref{eq:proof_lm41_6} and \eqref{eq:proof_lm41_6b} we obtain \eqref{eq:primal_recovery2}.
For the left-hand side inequality in the first estimate of \eqref{eq:primal_recovery2}, we use  a simple bound $-\norm{\xopt}\dist{\bb - \Ab\ub,\Kc} \leq \varphi(\ub) - \varphi^{\star}$ for $\ub = \bar{\ub}^k \in\Uc$ from the saddle point theory as in \eqref{eq:proof_lm41_4}.
\Eproof

\vspace{-4ex}
\bibliographystyle{plain}

\end{document}